\documentclass{amsart}

\usepackage{amssymb}
\usepackage{nicefrac}
\usepackage{enumitem}
\usepackage{hyperref}

\newcommand*{\mailto}[1]{\href{mailto:#1}{\nolinkurl{#1}}}

\newtheorem{theorem}{Theorem}[section]

\newtheorem{lemma}[theorem]{Lemma}
\newtheorem{corollary}[theorem]{Corollary}
\newtheorem{remark}[theorem]{Remark}
\newtheorem{hypothesis}[theorem]{Hypothesis}

\newcommand{\R}{{\mathbb R}}
\newcommand{\N}{{\mathbb N}}
\newcommand{\Z}{{\mathbb Z}}
\newcommand{\C}{{\mathbb C}}

\newcommand{\OO}{\mathcal{O}}
\newcommand{\oo}{o}

\newcommand{\nn}{\nonumber}

\newcommand{\ti}{\tilde}
\newcommand{\be}{\begin{equation}}
\newcommand{\ee}{\end{equation}}
\newcommand{\ba}{\begin{array}}
\newcommand{\ea}{\end{array}}

\newcommand{\spr}[2]{\langle #1 , #2 \rangle}

\newcommand{\id}{{\rm 1\hspace{-0.6ex}l}}
\newcommand{\E}{\mathrm{e}}
\newcommand{\I}{\mathrm{i}}

\newcommand{\im}{\mathrm{Im}}
\newcommand{\re}{\mathrm{Re}}
\newcommand{\dom}{\mathfrak{D}}

\newcommand{\floor}[1]{\lfloor#1 \rfloor}
\newcommand{\ceil}[1]{\lceil#1 \rceil}

\newcommand{\eps}{\varepsilon}
\newcommand{\vphi}{\varphi}
\newcommand{\sig}{\sigma}
\newcommand{\lam}{\lambda}

\numberwithin{equation}{section}

\begin{document}

\title[Singular Weyl--Titchmarsh--Kodaira theory]{Singular Weyl--Titchmarsh--Kodaira theory for one-dimensional Dirac operators}

\author[R.\ Brunnhuber]{Rainer Brunnhuber}
\address{Institute of Mathematics\\ University of Klagenfurt\\
Universit\"atsstrasse 65--67\\ 9020 Klagenfurt\\ Austria}
\email{\mailto{Rainer.Brunnhuber@aau.at}}
\urladdr{\url{http://www.aau.at/~rabrunnh/}}

\author[J.\ Eckhardt]{Jonathan Eckhardt}
\address{Institut Mittag-Leffler\\
Aurav\"agen 17\\ SE-182 60 Djursholm\\ Sweden}
\email{\mailto{jonathaneckhardt@aon.at}}

\author[A.\ Kostenko]{Aleksey Kostenko}
\address{Faculty of Mathematics\\ University of Vienna\\
Oskar-Morgenstern-Platz 1\\ 1090 Wien\\ Austria}
\email{\mailto{duzer80@gmail.com};\mailto{Oleksiy.Kostenko@univie.ac.at}}

\author[G.\ Teschl]{Gerald Teschl}
\address{Faculty of Mathematics\\ University of Vienna\\
Oskar-Morgenstern-Platz 1\\ 1090 Wien\\ Austria\\ and International
Erwin Schr\"odinger
Institute for Mathematical Physics\\ Boltzmanngasse 9\\ 1090 Wien\\ Austria}
\email{\mailto{Gerald.Teschl@univie.ac.at}}
\urladdr{\url{http://www.mat.univie.ac.at/~gerald/}}

\thanks{Monatsh. Math. {\bf 174}, 515--547 (2014)}
\thanks{{\it Research supported by the Austrian Science Fund (FWF) under Grant No.\ Y330 and M1309 as well as by the AXA Mittag-Leffler Fellowship Project, funded by the AXA Research Fund}}

\keywords{Dirac operators, spectral theory, Borg--Marchenko theorem}
\subjclass[2010]{Primary 34B20, 34L40; Secondary 34L10, 34A55}

\begin{abstract}
We develop singular Weyl--Titchmarsh--Kodaira theory for one-dimensional Dirac operators. In particular, we establish
existence of a spectral transformation as well as local Borg--Marchenko and Hochstadt--Lieberman type
uniqueness results. Finally, we give some applications to the case of radial Dirac operators.
\end{abstract}

\maketitle

\section{Introduction}
\label{sec:int}

The main aim of the present paper is to develop singular Weyl--Titchmarsh--Kodaira theory for one-dimensional Dirac operators.
Classical Weyl--Titchmarsh--Kodaira theory has originally been developed for one-dimensional Schr\"odinger operators
with one regular endpoint and has subsequently been extended to a number of other operators. 
For example, this has been done by Hinton and Shaw in a series of papers \cite{HS81}, \cite{HS82}, \cite{HS83}, \cite{HS84}, \cite{HS86} (see also \cite{HS93}, \cite{HS97} and the references in \cite{cg}) for general singular Hamiltonian systems which also include one-dimensional Dirac operators as a special case.  
However, it has been shown by Kodaira \cite{ko}, Kac \cite{ka} and more recently by Fulton \cite{ful08},
Gesztesy and Zinchenko \cite{gz}, Fulton and Langer \cite{fl}, Kurasov and Luger \cite{kl}, and Kostenko, Sakhnovich, and Teschl
\cite{kst}, \cite{kst2}, \cite{kst3}, \cite{kt}, \cite{kt2} that many aspects of this classical theory still can be established at a
singular endpoint. It has recently proven to be a powerful tool for inverse spectral theory for these operators
and further refinements were given by some of us in \cite{je}, \cite{je2}, \cite{egnt2}, \cite{egnt3}, \cite{et}, \cite{kt}. The analogous
theory for one-dimensional Dirac operators is still missing and it is the purpose of the present paper to fill this gap. 

As our first main result we establish existence of a spectral measure and the corresponding spectral transform (Theorem~\ref{thm:sptr}).
Furthermore, we prove a local Borg--Marchenko \cite{borg,mar} result (Theorem~\ref{thmbm}), which generalizes the classical result
whose local version was first established by Clark and Gesztesy \cite{cg} (see also \cite{sa02, sa06,sakb}). Next, in Section~\ref{secPertRD}, we apply our
results to radial Dirac operators. Namely, we show that for this class of Dirac operators the singular Weyl function is a generalized
Nevanlinna function (Theorem~\ref{thmpbgn}) and prove a local Borg--Marchenko result (Theorem~\ref{thmBMPRD}). Finally, we
show that in the case of purely discrete spectra the spectral measure uniquely determines the operator (Theorem~\ref{thmSpectFuncDisc})
and use this to establish a general Hochstadt--Lieberman-type uniqueness result (Theorem~\ref{thmHL}). An alternate
approach using the theory of de Branges spaces will be given in \cite{ekt} and spectral asymptotics for the singular Weyl
functions will be given in \cite{ekt2}.

For closely related research we also refer to \cite{egnt1,egnt2, egnt3, et2,fll,gkm,ma,ma2}.

\section{Singular Weyl--Titchmarsh--Kodaira theory}
\label{sec:swm}

We will be concerned with Dirac operators in the Hilbert space $L^2(I,\C^2)$, where $I=(a,b) \subseteq \R$ (with $-\infty \le
a < b \le \infty$) is an arbitrary interval. To this end, we consider the differential expression
\begin{equation}\label{dirac}
\tau = \frac{1}{\I} \sig_2 \frac{d}{dx} + Q(x).
\end{equation}
Here the potential matrix $Q(x)$ is given by
\begin{equation}
Q(x) =  q_{\rm el}(x)\id  + q_{\rm am}(x)\sig_1 +
(m+ q_{\rm sc}(x)) \sig_3,
\end{equation}
$\sig_1$, $\sig_2$, $\sig_3$ denote the Pauli matrices
\begin{equation}
\sig_1=\begin{pmatrix} 0 & 1 \\ 1 & 0\end{pmatrix}, \quad
\sig_2=\begin{pmatrix} 0 & -\I \\ \I & 0\end{pmatrix}, \quad
\sig_3=\begin{pmatrix} 1 & 0 \\ 0 & -1\end{pmatrix},
\end{equation}
and $m$, $q_{\rm sc}$, $q_{\rm el}$, and $q_{\rm am}$ are interpreted as
mass, scalar potential, electrostatic potential, and anomalous
magnetic moment, respectively (see \cite[Chapter~4]{th}).
As usual, we require that $m\in[0,\infty)$ and that $q_{\rm sc}$, $q_{\rm el}$, $q_{\rm
am}
\in L^1_{loc}(I)$ are real-valued. 

We do not include a
magnetic moment $\tilde{\tau} = \tau +\sig_2 q_{\rm mg}(x)$
as it can be easily eliminated by a simple gauge transformation
$\tau = \Gamma^{-1} \tilde{\tau} \Gamma$, where $\Gamma =\exp(-\I\int^x q_{\rm mg}(r) dr)$.
Furthermore, we will occasionally omit the electrostatic potential $q_{\rm el}$ since by employing the gauge transformation
\begin{align}\label{eqnGaugeEL}
\Gamma & = \begin{pmatrix} \cos(\varphi) & -\sin(\varphi)\\
\sin(\varphi) & \cos(\varphi) \end{pmatrix}, &
\varphi(x) & = \int^x q_{\rm el}(r) dr,
\end{align}
it is possible to transform the differential expression $\tau$ to the new form
\begin{align}
q_{\rm el} &\to 0,\\
q_{\rm am} &\to q_{\rm am} \cos(2\vphi) - (m+q_{\rm sc}) \sin(2\vphi), \\
m+ q_{\rm sc} &\to (m+q_{\rm sc}) \cos(2\vphi) + q_{\rm am} \sin(2\vphi).
\end{align}
Let us also mention that there is another gauge transformation which gets rid of $q_{\rm am}$
under some additional assumptions (see \cite[Section~7.1.1]{ls}).

Finally, if one solution $u$ is known, a second solution $v$ can be found using d'Alembert reduction (cf., e.g., \cite[Section~3.4]{tode}). 
In fact, if $\tau u= z u$, then
\begin{align}
v(x) = u(x) \int^x \frac{Q_{22}(r)-z}{u_1(r)^2} dr - \begin{pmatrix} 0 \\ u_1(x)^{-1} \end{pmatrix}
\end{align}
is a second solution with $W(v,u)=1$. Similarly,
\begin{align}
\ti{v}(x) = -u(x) \int^x \frac{Q_{11}(r)-z}{u_2(r)^2} dr - \begin{pmatrix} u_2(x)^{-1}\\0 \end{pmatrix}
\end{align}
is a second solution with $W(\ti{v},u)=1$ as well.

If $\tau$ is in the limit point case at both $a$ and $b$, then $\tau$ gives rise to a unique
self-adjoint operator $H$ when defined maximally (cf., e.g., \cite{ls},
\cite{wdl}, \cite{wei2}). Otherwise, we fix a boundary
condition at each endpoint where $\tau$ is in the limit circle case.
Explicitly, such an operator $H$ is given by
\begin{align}\begin{split}
H: \ba[t]{lcl} \dom(H) &\to& L^2(I,\C^2) \\ f &\mapsto& \tau f \ea
\end{split}\end{align}
where
\begin{align}\begin{split} \label{domH}
\dom(H) = \{ f \in L^2(I,\C^2) \,|\, &  f \in AC_{loc}(I,\C^2), \, \tau f \in
L^2(I,\C^2),\\ & \qquad\qquad W_a(u_-,f) = W_b(u_+,f) =0 \},
\end{split}\end{align}
with
\begin{equation}
W_x(f,g) = \I \spr{f^*(x)}{\sig_2 g(x)} = f_1(x) g_2(x) - f_2(x) g_1(x)
\end{equation}
the usual Wronskian (we remark that the limit $W_{a,b}(.,..) = \lim_{x \to
a,b} W_x(.,..)$ exists for functions as in (\ref{domH})). Here the function
$u_-$ (resp.\ $u_+$) used to generate the boundary condition at $a$
(resp.\ $b$) can be chosen to be a nontrivial solution of $\tau u =0$ if $\tau$
is in the limit circle case at $a$ (resp.\ $b$) and zero else.

For a given point $c\in I$ consider the operators $H^D_{(a,c)}$ and $H^D_{(c,b)}$
which are obtained by restricting $H$ to $(a,c)$ and $(c,b)$ with a Dirichlet boundary condition $f_1(c)=0$ at
$c$, respectively. The corresponding operators with a Neumann boundary condition $f_2(c)=0$ will be
denoted by $H^N_{(a,c)}$ and $H^N_{(c,b)}$.

Moreover, let $c(z,x)$ and $s(z,x)$ be the solutions of $\tau u = z\, u$ corresponding
to the initial conditions $c_1(z,c)=1$, $c_2(z,c)=0$ and $s_1(z,c)=0$, $s_2(z,c)=1$.
Then we can define the Weyl solutions
\begin{align}
u_-(z,x) &= c(z,x) - m_-(z) s(z,x), \qquad z\in\C\setminus\sig(H^D_{(a,c)}),\\\label{defupm}
u_+(z,x) &= c(z,x) + m_+(z) s(z,x), \qquad z\in\C\setminus\sig(H^D_{(c,b)}),
\end{align}
where $m_\pm(z)$ are the Weyl $m$-functions corresponding to the base point $c$ and associated with
$H^D_{(a,c)}$, $H^D_{(c,b)}$, respectively. Note that the functions $m_\pm(z)$ are Herglotz--Nevanlinna functions.

We refer to the monographs \cite{ls}, \cite{wdl}, \cite{wei2} for background
and also to \cite{th} for further information about Dirac operators and their
applications.

For the rest of this section we will closely follow the presentation from \cite{kst2}. Most proofs can be done literally following the
arguments in \cite{kst2} and hence we will omit them here. Our first ingredient to define an analogous singular Weyl function
at $a$ is a system of real entire solutions $\Theta(z,x)$ and $\Phi(z,x)$ such that $\Phi(z,x)$ lies in the domain of
$H$ near $a$ and such that the Wronskian satisfies $W(\Theta(z),\Phi(z))=1$. To this end, we require the following
hypothesis, which turns out necessary and sufficient for such a system of solutions to exist.

\begin{hypothesis}\label{hyp:gen}
Suppose that the spectrum of $H^D_{(a,c)}$ is purely discrete.
\end{hypothesis}

\begin{lemma}\label{lem:pt}
The following properties are equivalent:
\begin{enumerate}[label=\emph{(\roman*)}, leftmargin=*, widest=XX]
\item The spectrum of $H^D_{(a,c)}$ is purely discrete.
\item There is a real entire solution $\Phi(z,x)$, which is non-trivial and lies in the domain of $H$
near $a$ for each $z\in\C$.
\item There are real entire solutions $\Theta(z,x)$, $\Phi(z,x)$ with $W(\Theta(z),\Phi(z))=1$, such that
 $\Phi(z,x)$ is non-trivial and lies in the domain of $H$ near $a$ for each $z\in\C$.
\end{enumerate}
\end{lemma}

Given such a system of real entire solutions $\Theta(z,x)$ and $\Phi(z,x)$, we define the singular Weyl function
\begin{equation}\label{defM}
M(z) = -\frac{W(\Theta(z),u_+(z))}{W(\Phi(z),u_+(z))}
\end{equation}
such that the solution lying in the domain of $H$ near $b$ is given by
\begin{equation}
u_+(z,x)= \alpha(z) \big(\Theta(z,x) + M(z) \Phi(z,x)\big), \quad x\in I,
\end{equation}
where $\alpha(z)= - W(\Phi(z),u_+(z))$.
It is immediate from the definition that the singular Weyl function $M(z)$ is analytic in $\C\backslash\R$
and satisfies $M(z)=M(z^*)^*$. Note that $M(z)$ will in general not be a Herglotz--Nevanlinna function.
However, following literally the argument in \cite[Lemma~3.2]{kst2}, one infers that associated with $M(z)$ is a corresponding
spectral measure $\rho$ given by the Stieltjes--Liv\v{s}i\'{c} inversion formula
\begin{equation}\label{defrho}
\frac{1}{2} \left( \rho\big((\lambda_0,\lambda_1)\big) + \rho\big([\lambda_0,\lambda_1]\big) \right)=
\lim_{\eps\downarrow 0} \frac{1}{\pi} \int_{\lambda_0}^{\lambda_1} \im\big(M(\lambda+\I\eps)\big) d\lambda.
\end{equation}

\begin{theorem}\label{thm:sptr}
Suppose Hypothesis~\ref{hyp:gen} and let the spectral measure $\rho$ be given by \eqref{defrho}.
The mapping
\begin{equation}
U: L^2(I, \C^2) \to L^2(\R,d\rho), \qquad f \mapsto \hat{f}
\end{equation}
where $\hat{f}$ is defined by
\begin{equation}\label{eqhatf}
\hat{f}(\lam) = \lim_{c\uparrow b} \int_a^c \Phi_1(\lam,x) f_1(x) + \Phi_2(\lam,x) f_2(x)~dx
\end{equation}
is unitary and its inverse
\begin{equation}
U^{-1}: L^2(\R,d\rho) \to L^2(I, \C^2), \qquad \hat{f} \mapsto f
\end{equation}
is given by
\begin{equation}\label{Uinv}
f(x) = \lim_{r\to\infty} \int_{-r}^r \Phi(\lam,x) \hat{f}(\lam)d\rho(\lam) =
\lim_{r\to\infty} \begin{pmatrix} \int_{-r}^r \Phi_1(\lam,x) \hat{f}(\lam)~d\rho(\lam) \\
\int_{-r}^r \Phi_2(\lam,x) \hat{f}(\lam)~d\rho(\lam)\end{pmatrix}.
\end{equation}
Moreover, $U$ maps $H$ to multiplication by $\lam$. Note that the right-hand sides of
\eqref{eqhatf} and \eqref{Uinv}  are to be understood as limits in $L^2(\R,d\rho)$ and
$L^2(I, \C^2)$, respectively.
\end{theorem}

\begin{corollary}\label{corsup}
The sets
\begin{align} 
\Sigma_{ac} &= \{\lam \,|\, 0<\limsup_{\eps\downarrow 0} \im(M(\lam+\I\eps)) < \infty\},\\ 
\Sigma_s &= \{\lam \,| \limsup_{\eps\downarrow 0}\im(M(\lam+\I\eps)) = \infty\},\\ 
\Sigma_p &= \{\lam \,| \lim_{\eps\downarrow 0} \eps\im(M(\lam+\I\eps))>0 \},\\ 
\Sigma &= \Sigma_{ac} \cup \Sigma_s = \{\lam \,|\, 0<\limsup_{\eps\downarrow 0} \im(M(\lam+\I\eps))\}
\end{align}
are minimal supports for $\rho_{ac}$, $\rho_s$, $\rho_{pp}$ and $\rho$, respectively.
We could even restrict ourselves to values of $\lam$
where the $\limsup$ is a $\lim$ (finite or infinite).\newline
Moreover, the spectrum of $H$ is given by the closure of $\Sigma$,
\be
\sigma(H) = \overline{\Sigma},
\ee
the point spectrum (the set of eigenvalues) is given by $\Sigma_p$,
\be
\sigma_p(H) = \Sigma_p,
\ee
and the absolutely continuous spectrum is given by the essential closure
of $\Sigma_{ac}$,
\be
\sigma(H_{ac}) = \overline{\Sigma}_{ac}^{ess}.
\ee
Recall that $\overline{\Omega}^{ess} = \{ \lam\in\R \,|\,
|(\lam-\eps,\lam+\eps)\cap \Omega|>0 \mbox{ for all } \eps>0\}$ where $|\Omega|$ denotes
the Lebesgue measure of a Borel set $\Omega$.
\end{corollary}

Rather than $u_+(z,x)$, we will use
\begin{equation}\label{defpsi}
\Psi(z,x)= \Theta(z,x) + M(z) \Phi(z,x), \quad x\in I.
\end{equation}
For the resolvent we have
\be \label{gfres}
(H-z)^{-1} f(x) = \int_a^b G(z,x,y) f(y) dy,
\ee
where
\be\label{defgf}
G(z,x,y) = \begin{cases} \Psi(z,x)\otimes\Phi(z,y), & y<x,\\ \Phi(z,x) \otimes \Psi(z,y), & y>x,
\end{cases}
\ee
is the Green's function of $H$.

We conclude this section with a simple fact concerning the spectral transformation of the
Green's function of $H$ which will turn out to be useful later on.

\begin{lemma}\label{lemUub}
Recall the Green's function
\be 
G(z,x,y) = \begin{pmatrix}G_{11}(z,x,y) & G_{12}(z,x,y) \\ G_{21}(z,x,y) & G_{22}(z,x,y)  \end{pmatrix}
\ee
of $H$ defined in \eqref{defgf}. Then we have
\be\label{UG}
(U G_{i}(z,x,.))(\lam) = \frac{\Phi_i(\lam,x)}{\lam-z}, \quad\quad\quad\text{i=1,2}
\ee
for every $x\in I$ and every $z\in\C\setminus\sigma(H)$. Here $G_{i}(z,x,.)$ has to be interpreted as
\be \label{greennotation}
G_{i}(z,x,.) = \begin{pmatrix}G_{i1}(z,x,.) \\G_{i2}(z,x,.)\end{pmatrix}, \quad\quad\quad\text{i=1,2}.
\ee
\end{lemma}

\begin{proof}
First we observe that, by \eqref{defgf}, $G_i(z,x,.)\in L^2(I, \C^2)$, $i=1,2$,
for every $x\in I$ and every $z\in\C\setminus\sigma(H)$. Moreover, we have
\[
(H-z)^{-1} f = U^{-1} \frac{1}{\lam-z} U f
\]
where the left-hand side is given by \eqref{gfres} and the right-hand side can be written as
\[
\lim_{c\uparrow b} \int_a^c \frac{\Phi(\lam,x)}{\lam-z} \hat{f}(\lam)d\rho(\lam).
\]
Hence both sides are equal in $L^2(I,\C^2)$ and hence in particular for almost every $x\in I$.
Moreover, if $\hat{f}$ has compact support we can drop the limit and both sides are continuous
with respect to $x$, showing equality for all $x\in I$ in this case. Since such $f$ are dense the
claim follows.
\end{proof}

Differentiating with respect to $z$, we get:

\begin{corollary}
We even have \label{corGdz}
\be\label{UGdz}
(U \partial_z^k G_{i}(z,x,.))(\lam) = \frac{k! \Phi_i(\lam,x)}{(\lam-z)^{k+1}}, \quad\quad\quad\text{i=1,2}
\ee
for every $x\in I$, $k \in \N_0$ and $z\in\C\setminus\sigma(H)$.
\end{corollary}

\begin{proof}
We prove the claim by induction. For the case $k=0$, \eqref{UGdz} is just \eqref{UG}. For $k=1$  we have
\begin{align*}
\partial_z(UG_{i}(z,x,.))(\lam) = \partial_z \left(\frac{\Phi_i(\lam,x)}{\lam-z}\right) = \frac{\Phi_i(\lam,x)}{(\lam-z)^2},
\quad \quad \quad i=1,2,
\end{align*}
where
\begin{align*}
\partial_z(UG_{i}(z,x,.))(\lam) &=\int_a^b \Phi_1(\lam,y) \partial_z G_{i1}(z,x,y) +  \Phi_2(\lam,y) \partial_z G_{i2}(z,x,y)~dy \\
&= (U \partial_z G_{i}(z,x,.))(\lam), \quad \quad \quad i=1,2. 
\end{align*}
Now suppose \eqref{UGdz} holds for $k=n$. Then we have
\be \nn
\partial_z(U\partial_z^nG_{i}(z,x,.))(\lam) = \partial_z \left(\frac{n! \Phi_i(\lam,x)}{(\lam-z)^{n+1}}\right)
= \frac{(n+1)! \Phi_i(\lam,x)}{(\lam-z)^{n+2}}, \quad \quad \quad i=1,2
\ee
where $\partial_z(U\partial_z^nG_{i}(z,x,.))(\lam)=(U\partial_z^{n+1}G_{i}(z,x,.))(\lam)$, $i=1,2$
by performing the same computation as above with $U\partial_z^nG_{i}$ instead of $UG_{i}$.
Thus we have verified \eqref{UGdz} for every $k \in \N_0$.
\end{proof}

\begin{remark}\label{rem:uniq}
It is important to point out that a fundamental system $\Theta(z,x)$, $\Phi(z,x)$ of solutions is not unique and any other such system is given by
\begin{align}
\ti{\Theta}(z,x) & = \E^{-g(z)} \Theta(z,x) - f(z) \Phi(z,x), &
\ti{\Phi}(z,x) & = \E^{g(z)} \Phi(z,x),
\end{align}
where $f(z)$, $g(z)$ are entire functions with $f(z)$ real and $g(z)$ real modulo $\I\pi$.
The singular Weyl functions are related via
\begin{align}
\ti{M}(z) = \E^{-2g(z)} M(z) + \E^{-g(z)}f(z)
\end{align}
and the corresponding spectral measure is given by
\begin{align}
d\ti{\rho}(\lam) = \E^{-2g(\lam)} d\rho(\lam).
\end{align}
In particular, the two measures are mutually absolutely continuous and the associated spectral
transformations just differ by a simple rescaling with the positive function $\E^{-2g(\lam)}$.
\end{remark}

Next, the following integral representation shows that $M(z)$ can be reconstructed from $\rho$ up to an entire function.

\begin{theorem}[\cite{kst2}]\label{IntR}
Let $M(z)$ be a singular Weyl function and $\rho$ its associated spectral measure. Then there exists
an entire function $g(z)$ such that $g(\lam)\ge 0$ for $\lam\in\R$ and $\E^{-g(\lam)}\in L^2(\R, d\rho)$. 
Moreover, for any entire function $\hat{g}(z)$ such that $\hat{g}(\lam)>0$ for $\lam\in\R$ and $(1+\lam^2)^{-1} \hat{g}(\lam)^{-1}\in L^1(\R, d\rho)$
(e.g.\ $\hat{g}(z)=\E^{2g(z)}$) we have the integral representation
\begin{equation}\label{Mir}
M(z) = E(z) + \hat{g}(z) \int_\R \left(\frac{1}{\lam-z} - \frac{\lam}{1+\lam^2}\right) \frac{d\rho(\lam)}{\hat{g}(\lam)},
\qquad z\in\C\backslash\sig(H),
\end{equation}
where $E(z)$ is a real entire function.
\end{theorem}

\begin{remark}\label{rem:herg}
Choosing a real entire function $g(z)$ such that $\exp(-2g(\lam)) \in L^1(\R, d\rho)$ we see that
\begin{align}
M(z) = \E^{2g(z)} \int_\R \frac{1}{\lam-z} \E^{-2g(\lam)}d\rho(\lam) - E(z)
\end{align}
for some real entire function $E(z)$.
Hence if we choose $f(z) = \exp(-g(z)) E(z)$ and switch to a new system of solutions as in Remark~\ref{rem:uniq},
then we see that the new singular Weyl function is a Herglotz--Nevanlinna function
\begin{align}
 \ti{M}(z) = \int_\R \frac{1}{\lam-z} \E^{-2g(\lam)}d\rho(\lam).
\end{align}
\end{remark}

As another consequence we get a criterion when our singular Weyl function is a
generalized Nevanlinna function with no nonreal poles and the only generalized pole of nonpositive type at $\infty$.
We will denote the set of all such generalized Nevanlinna functions by $N_\kappa^\infty$.

\begin{theorem}[\cite{kst2}]\label{thm:nkap}
Fix the solution $\Phi(z,x)$. Then there is a corresponding solution $\Theta(z,x)$ such that $M(z)\in N_\kappa^\infty$
for some $\kappa\le k$ if and only if $(1+\lam^2)^{-k-1} \in L^1(\R,d\rho)$. Moreover, $\kappa=k$ if $k=0$ or
$(1+\lam^2)^{-k} \not\in L^1(\R,d\rho)$.
\end{theorem}

In order to identify possible values of $k$ one can try to bound $\lam^{-k}$ by
linear combinations of $\Phi_1(\lam,x)^2$ and $\Phi_2(\lam,x)^2$ which are in $L^1(\R, (1+\lam^2)^{-1}d\rho)$ by
Lemma~\ref{lemUub}.

As a final ingredient we will need the following simple lemma on high energy asymptotics of our real entire solution $\Phi(z,x)$.

\begin{lemma}
If $\Phi(z,x)$ is a real entire solution which lies in the domain of $H$ near $a$, then for every $x_0$, $x\in I$
\be\label{IOphias}
\Phi(z,x) = \Phi(z,x_0) \E^{-\I (x-x_0) z + \I \int_{x_0}^x q_{\rm el}(r)dr} (1 + \oo(1)),
\ee
as $\im(z)\to\infty$.
\end{lemma}

\begin{proof}
Using
\[
\Phi_1(z,x) = \Phi_1(z,c) ( c_1(z,x) - m_-(z) s_1(z,x) )
\]
and the well-known asymptotics (cf.\ \cite{cg}, \cite[Lemma~7.2.1]{ls} --- in fact this also follows as a special case from
Lemma~\ref{lem:asphi} below)
\begin{align*}
c(z,x) &= \begin{pmatrix}\cos(z (x-c) - \int_c^x q_{\rm el}(r)dr) \\ -\sin(z (x-c) - \int_c^x q_{\rm el}(r)dr) \end{pmatrix} + \oo\big(\E^{|\im(z)|(x-c)}\big),\\
s(z,x) &= \begin{pmatrix}\sin(z (x-c) - \int_c^x q_{\rm el}(r)dr) \\ \cos(z (x-c) - \int_c^x q_{\rm el}(r)dr) \end{pmatrix} + \oo\big(\E^{|\im(z)|(x-c)}\big)
\end{align*}
for $x>c$ and
\[
m_-(z) = \I + \oo(1)
\]
we see~\eqref{IOphias} for $x_0=c$ and $x>c$. The second component follows similarly from
$\Phi_2(z,x) = \Phi_2(z,c) ( s_2(z,x) - m_-(z)^{-1} c_2(z,x) )$. The case $x<x_0$ follows after reversing the roles of $x_0$ and $x$.
Since $c$ is arbitrary, the proof is complete.
\end{proof}

\section{Supersymmetry}

In this section we want to establish the connection with the standard theory for one-dimensional Schr\"odinger operators
(cf.\ \cite{kst2}) if our Dirac operator is supersymmetric, that is, $q_{\rm el}= q_{\rm sc}=0$. In this case we can write our operator
as
\be\label{eq:dsusy}
H = \begin{pmatrix}
m & A_q\\ A_q^* & -m
\end{pmatrix}
\ee
where
\begin{align}\begin{split}
A_q f &= a_q f, \qquad a_q = - \frac{d}{dx} + q_{\rm am}(x), \\
\dom(A_q) &= \{ f\in L^2(a,b) \,|\, f\in AC_{loc}(a,b),\: a_q f \in L^2(a,b) \}.
\end{split}\end{align}
Here we use $A_q$ and $a_q$ for the operator and differential expression, respectively.
It is straightforward to check (cf.\ \cite[Problem~9.3]{tschroe}) that $A_q$ is closed and that its adjoint
is given by
\begin{align}\begin{split}
A_q^* f &=  a_q^* f, \qquad a_q^* = \frac{d}{dx} + q_{\rm am}(x), \\
\dom(A_q^*) & = \{ f\in L^2(a,b) \,|\,  f\in AC_{loc}(a,b),\: a_q^* f \in L^2(a,b),\\
&\qquad\qquad\qquad\qquad\qquad\qquad \lim_{x\to a,b} f(x) g(x) =0, \forall g\in\dom(A_q) \}.\label{eq:a*}
\end{split}\end{align}
In particular, our operator $H$ is self-adjoint. Note that if $\tau$ is in the limit point case at $a$ (or $b$), then
the boundary conditions in \eqref{eq:a*} hold automatically at the corresponding endpoint. However, in the limit circle case the operator $H$ is associated with a
specific boundary condition (for instance, in the case when both endpoint are regular, the boundary conditions in \eqref{eq:a*} are precisely the Dirichlet conditions). A straightforward computation verifies
\be\label{eqhsq}
H^2 = \begin{pmatrix}
A_q A_q^* +m^2 & 0\\ 0 & A_q^*A_q + m^2
\end{pmatrix}.
\ee
Here, $A_q^* A_q$ and $A_qA_q^*$ are generalized one-dimensional Schr\"odinger operators
of the type considered in \cite{egnt1,egnt2}.

Note that in this case $\tau u = z u$ is equivalent to
\be
a_q a_q^* u_1 = (z^2-m^2) u_1, \qquad u_2 = (z+m)^{-1} a_q^* u_1,
\ee
as well as
\be
a_q^* a_q u_2 = (z^2-m^2) u_2, \qquad u_1 = (z-m)^{-1} a_q u_2.
\ee
By spectral mapping, \eqref{eqhsq} implies that Hypothesis~\ref{hyp:gen} will hold if and only
if the corresponding hypothesis holds for $A_q^* A_q$ (or $A_qA_q^*$). Consequently,
Theorem~8.4 from \cite{egnt2} implies that there is a system of entire solutions $\phi(\zeta,x)$,
$\theta(\zeta,x)$ for $(a_q a_q^*-\zeta)y=0$ such that $\phi(\zeta,x)$ is in the domain of $A_qA_q^*$ near $a$. One easily checks that $a_q^* \phi(\zeta,x)$,
$\zeta^{-1} a_q^* \theta(\zeta,x)$ is a corresponding system for $(a_q^* a_q-\zeta)y=0$. Thus
\be\label{eq:ptsusy}
\Phi(z,x)= \begin{pmatrix} (z+m) \phi(z^2-m^2,x)\\ a_q^*\phi(z^2-m^2,x) \end{pmatrix}, \quad
\Theta(z,x)= \begin{pmatrix} \theta(z^2-m^2,x)\\ \frac{1}{z+m}a_q^* \theta(z^2-m^2,x) \end{pmatrix}
\ee
is a corresponding system for our Dirac operator $H$. Note that since $a_q^* \theta(0,x) =0$ the solution
$\Theta(z,x)$ is indeed entire. Moreover, by Theorem~3.4 of \cite{kst3},
\be
\Psi(z,x)= \Theta(z,x) + \frac{m_q(z^2-m^2)}{z+m} \Phi(z,x),
\ee
lies in the domain of $H$ near $b$ (here $m_q(\zeta)$ is the singular Weyl function of $A_q A_q^*$).
Note that while in \cite{kst3} we assumed $q_{\rm am} \in AC(a,b)$ the results extend to the present
situation in a straightforward manner.

We summarize our main findings from this section in the following theorem.

\begin{theorem}
Let our Dirac operator $H$ be given by \eqref{eq:dsusy} and suppose the spectrum of the Schr\"odinger-type operator
$A_q A_q^*$ is purely discrete when restricted to $(a,c)$. Then $H$ satisfies Hypothesis~\ref{hyp:gen} and the
singular Weyl function associated with the fundamental system \eqref{eq:ptsusy} is given by
\be
M(z)=\frac{m_q(z^2-m^2)}{z+m},
\ee
where $m_q(\zeta)$ is the singular Weyl function of $A_q A_q^*$.
\end{theorem}

We will use this connection to investigate an illustrative example in the next section.

\section{An example: The unperturbed radial Dirac operator}\label{sec4}

In this section we completely solve a prototypical example of a Dirac operator with two singular endpoints, namely
the unperturbed radial Dirac operator. We will use this explicit example to illustrate
some results from the foregoing sections. Additional information about the radial Dirac operator
mentioned in this section can be found in \cite{baev,gtv,th}. Explicitly, we look at the case
where the interval is the positive half-axis and the potential $Q(x)$ is given by
\begin{equation}\label{eq:Qkap}
Q_\kappa(x)= \begin{pmatrix}m & \frac{\kappa}{x} \\
\frac{\kappa}{x}  &-m\end{pmatrix}, \qquad x\in(0,\infty).
\end{equation}
As the case $\kappa<0$ can be reduced to the case $\kappa>0$ by the simple gauge transformation
$- \sigma_1 \tau \sigma_1$, we restrict our attention to the case $\kappa\ge 0$.

This particular Dirac operator is of the type considered in the previous section and hence can be reduced to the analysis of
the Bessel equation
\be\label{eq:bes}
-u''+\frac{l(l+ 1)}{x^2}u=\zeta u,\quad l\ge -\frac{1}{2}.
\ee
For the analysis of this equation in the context of singular Weyl--Titchmarsh--Kodaira theory we refer to (e.g.) \cite{kst}.
Here we just state the relevant results. Recall that a particular fundamental system of entire solutions of \eqref{eq:bes} satisfying
\be \label{a0}
W(\theta_l(\zeta),\phi_l(\zeta))=1
\ee
is given by
\be\label{defphil}
\phi_l(\zeta,x) = 
\zeta^{-\frac{2l+1}{4}} \sqrt{\frac{\pi x}{2}} J_{l+\frac{1}{2}}(\sqrt{\zeta} x),
\ee
\be\label{defthetal}
\theta_l(\zeta,x) = -
\zeta^{\frac{2l+1}{4}} \sqrt{\frac{\pi x}{2}} \begin{cases}
\frac{-1}{\sin((l+\frac{1}{2})\pi)} J_{-l-\frac{1}{2}}(\sqrt{\zeta} x), & {l+\frac{1}{2}} \in \R_+\setminus \N_0,\\
Y_{l+\frac{1}{2}}(\sqrt{\zeta} x) -\frac{1}{\pi}\log(\zeta) J_{l+\frac{1}{2}}(\sqrt{\zeta} x), & {l+\frac{1}{2}} \in\N_0,\end{cases}
\ee
where $J_{\nu}$ and $Y_{\nu}$ are the usual Bessel and Neumann functions \cite{dlmf}.
All branch cuts are chosen along the negative real axis unless explicitly stated otherwise.
If $\nu$ is an integer they of course reduce to spherical Bessel and Neumann functions
and can be expressed in terms of trigonometric functions (cf.\ e.g.\ \cite{dlmf}, \cite[Section~10.4]{tschroe}).
Finally, the Weyl solution is given by
\be
\psi_l(\zeta,x)= \theta_l(\zeta,x) + m_l(\zeta) \phi_l(\zeta,x)=
\I \sqrt{\frac{\pi x}{2}} (\I \sqrt{-\zeta})^{l+\frac{1}{2}} H_{l+\frac{1}{2}}^{(1)}(\I\sqrt{-\zeta}x)
\ee
and the singular Weyl function is
\be\label{eq:II.11}
m_l(\zeta) = \begin{cases}
\frac{-1}{\sin((l+\frac{1}{2})\pi)} (-\zeta)^{l+\frac{1}{2}}, & {l+\frac{1}{2}}\in\R_+\setminus \N_0,\\
\frac{-1}{\pi} z^{l+\frac{1}{2}}\log(-\zeta), & {l+\frac{1}{2}} \in\N_0,\end{cases}
\ee
where 
$H_{l+1/2}^{(1)} = J_{l+1/2} + \I Y_{l+1/2}$ are the Hankel functions of the first kind.

Using these formulae, and abbreviating
\begin{align}
a_\kappa^*  =\frac{d}{dx} +\frac{\kappa}{x}, \qquad \zeta  =z^2-m^2,
\end{align}
we immediately obtain the regular radial solution
\be\label{eq:phi_dirac}
\Phi_\kappa(z,x)=
\begin{pmatrix}
(z+m)\phi_\kappa(\zeta,x)\\
a_\kappa^*\phi_{\kappa}(\zeta,x)
\end{pmatrix}
\ee
and the singular radial solution
\be\label{eq:theta_dirac}
\Theta_\kappa(z,x)=
\begin{pmatrix}
\theta_\kappa(\zeta,x)\\
\frac{1}{z+m}a_\kappa^*\theta_{\kappa}(\zeta,x)
\end{pmatrix}.
\ee
Using \cite[formulas (5.5.3), (10.6.2)]{dlmf}, we obtain
\be\label{recphi}
a_\kappa^* \phi_\kappa(\zeta,x) = 
\zeta^{-\frac{2\kappa-1}{4}} \sqrt{\frac{\pi x}{2}} J_{\kappa-\frac{1}{2}}(\sqrt{\zeta} x)
=\begin{cases}
 \phi_{\kappa-1}(\zeta,x), & \kappa\ge \frac{1}{2},\\
\cos(\pi\kappa)\theta_{-\kappa}(\zeta,x), & \kappa\in [0,\frac{1}{2}),
 \end{cases}
\ee
as well as
\be\label{rectheta}
a_\kappa^*\theta_\kappa(\zeta,x)
=\begin{cases}
  \zeta
  \theta_{\kappa-1}(\zeta,x), & \kappa\ge \frac{1}{2},\\
\frac{\zeta}{\cos(\pi\kappa)}
\phi_{-\kappa}(\zeta,x), & \kappa\in [0,\frac{1}{2}).
 \end{cases}
\ee
By construction we have
\begin{align}
W_x(\Theta_\kappa(z),\Phi_k(z)) = 1
\end{align}
and our singular Weyl function defined by
\begin{align}
\Psi_\kappa(z,x)=\Theta_\kappa(z,x)+M_\kappa(z)\Phi_\kappa(z,x)\in L^2((1,\infty),\C^2)
\end{align}
is given by
\begin{align}
M_\kappa(z)=\frac{1}{z+m}m_\kappa(z^2-m^2),\quad z\in\C\setminus(-\infty,-m]\cup[m,\infty).
\end{align}
The associated spectral measure is given by
\be
d\rho_\kappa(\lam) = 
\chi_{(-\infty,-m]\cup[m,\infty)}(\lam) \frac{|\lam^2-m^2|^{\kappa+1/2}}{|\lam|+m} \frac{d\lam}{\pi}.
\ee
Furthermore, one infers that $M_\kappa(z)$ is in the generalized Nevanlinna class $N_{\kappa_0}^\infty$ with $\kappa_0=\floor{\kappa + 1/2}$.

\section{The limit circle case}

In this section we are going to extend \cite[Appendix A]{kst2} to the case of one-dimensional
Dirac operators. 
More precisely, we show that whenever $\tau$ is in the limit circle case at $a$, we may introduce a particular fundamental system such that the corresponding singular Weyl function is a Herglotz--Nevanlinna function. 
The proofs only require straightforward adaptations and we hence omit them here (details can be found in \cite{brun}).

To this end, we start with a Hypothesis which will turn out to be
equivalent to the claim that $\tau$ is in the limit circle case at $a$.
\begin{hypothesis}\label{hyp:limit} Fix $\lam_0 \in \R$ and suppose that $\Phi_0(x)$ and $\Theta_0(x)$ 
are two real-valued solutions of $\tau u =\lam_0 u$
which satisfy $W(\Theta_0, \Phi_0)=1$. Assume that the limits
\be \label{wronskilimits}
\lim_{x \rightarrow a} W_x(\Phi_0, u(z)) \qquad \text{and} \qquad \lim_{x \rightarrow a} W_x(\Theta_0, u(z))
\ee
exist for every solution $u(z,x)$ of $\tau u = z u$.
\end{hypothesis}

\begin{remark} \label{hypindependent}
Hypothesis \ref{hyp:limit} is independent of the choice of $\lam_0\in\R$.
\end{remark}

Indeed, let $\Phi_1(x)$ and $\Theta_1(x)$ be two real-valued solutions of $\tau u = \lam_1 u$ for some
$\lam_1 \in \R$ which satisfy $W(\Theta_1, \Phi_1)=1$. Setting $f_1=\Phi_0(x)$,
$f_2 = \Phi_1(x)$, $f_3=\Theta_0(x)$ and $f_4=u(z,x)$ in the Pl\"ucker identity
\be \label{pluecker}
W_x (f_1,f_2)W_x(f_3, f_4) + W_x(f_1, f_3)W_x(f_4, f_2) + W_x(f_1, f_4) W_x(f_2, f_3) = 0
\ee
and using $W(\Theta_0, \Phi_0)=1$ yields
\be \label{lambda:ind}
W_x(\Phi_1, u(z))= W_x(\Phi_0,\Phi_1)W_x (\Theta_0, u(z)) - W_x(\Phi_0, u(z)) W_x(\Theta_0, \Phi_1).
\ee
The Pl\"ucker identity \eqref{pluecker} remains valid in the limit $x \rightarrow a$. If
Hypothesis \ref{hyp:limit} holds, the limit $\lim_{x\rightarrow a}W_x(\Phi_1, u(z))$ exists, as then
all limits on the right-hand side of \eqref{lambda:ind} exist.
To see that $\lim_{x\rightarrow a}W_x(\Theta_1, u(z))$ exists as well, one just needs to replace
$\Phi_1(x)$ by $\Theta_1(x)$ in the above calculation.
Altogether we have shown that, if Hypothesis~\ref{hyp:limit} holds for one $\lam_0 \in \R$, then it also holds for
any other $\lam_1 \in \R$ which justifies Remark~\ref{hypindependent}.

\begin{lemma} \label{limitcircle-hyp}
If $\tau$ is in the limit circle case at $a$, then Hypothesis \ref{hyp:limit} holds. In this case, the limits
\eqref{wronskilimits} are holomorphic with respect to $z$ whenever $u(z,x)$ is.
\end{lemma}

Now suppose $\tau$ satisfies Hypothesis \ref{hyp:limit} and set
\begin{align}
\Phi(z,x)&=W_a(c(z),\Phi_0)s(z,x) - W_a(s(z),\Phi_0)c(z,x),  \label{limitsol1} \\
\Theta(z,x)&=W_a(c(z),\Theta_0)s(z,x) - W_a(s(z),\Theta_0)c(z,x). \label{limitsol2}
\end{align}
Hereby, the solutions $s(z,x)$ and $c(z,x)$ are defined in the same way as in Section~\ref{sec:swm}. Observe that we
have $\Phi(z,x)^*=\Phi(z^*,x)$ and $\Theta(z,x)^*=\Theta(z^*,x)$. Moreover, an easy calculation shows
$\Phi(\lam_0,x)= \Phi_0(x)$ and $\Theta(\lam_0,x)= \Theta_0(x)$.

\begin{lemma} \label{wronskilemma}
If Hypothesis \ref{hyp:limit} holds, then the solutions $\Phi(z,x)$ and $\Theta(z,x)$ defined in
\eqref{limitsol1} and \eqref{limitsol2} satisfy $W(\Theta(z),\Phi(z))=1$ as well as the identities
\begin{align}
W_a(\Theta(z),\Phi(\hat{z})) &=1, &
W_a(\Phi(\hat{z}), \Phi(z)) = W_a(\Theta(\hat{z}), \Theta(z)) & = 0.
\end{align}
\end{lemma}

Now we will prove that Hypothesis \ref{hyp:limit} is in fact equivalent to $\tau$ being in the limit circle case at $a$.

\begin{corollary}
If Hypothesis \ref{hyp:limit} holds, then $\tau$ is in the limit circle case at $a$. Moreover,
the solutions $\Phi(z,x)$ and $\Theta(z,x)$ defined in
\eqref{limitsol1} and \eqref{limitsol2} satisfy
\begin{align}
W_c(\Phi(z)^*, \Phi(z)) &= -  2\I\,\im (z) \int_a^c | \Phi(z,x) |^2 dx, \label{wronskiintphi} \\
W_c(\Theta(z)^*, \Theta(z)) &= -2\I\,\im (z) \int_a^c | \Theta(z,x) |^2 dx, \label{wronskiinttheta}
\end{align}
and are entire with respect to $z$.
\end{corollary}

\begin{lemma}
Suppose Hypothesis \ref{hyp:limit}, let $H$ be some self-adjoint operator associated with $\tau$ and
let the boundary condition at $a$ be induced by $\Phi_0$. 
Then $\Phi(z,x)$ defined in \eqref{limitsol1} lies in the domain of $H$ near $a$. Moreover,
we have
\be \label{weyl-}
m_{-}(z)= \frac{W_a(\Phi_0, c(z))}{W_a(\Phi_0, s(z))}.
\ee
\end{lemma}

We are now able to introduce the singular Weyl function $M(z)$ associated with the solutions $\Phi(z,x)$ and $\Theta(z,x)$.
As in Section~\ref{sec:swm}, this is done by requiring that
\be \label{weyllimit}
\Psi(z,x) = \Theta(z,x) + M(z) \Phi(z,x) \in L^2((c,b), \C^2)
\ee
and that $\Psi(z,x)$ satisfies the boundary condition of $H$ at $b$ if $\tau$ is limit circle at $b$. 
The following theorem contains the main result of this section.

\begin{theorem} \label{limitthm}
Suppose Hypothesis \ref{hyp:limit}, let $H$ be some self-adjoint operator
associated with $\tau$ and let the boundary condition at $a$ be induced by $\Phi_0$.
Then the singular Weyl function defined in \eqref{weyllimit} is a Herglotz--Nevanlinna
function and satisfies
\be \label{weylfunclimit1}
\im (M(z))= \im (z) \int_a^b | \Psi (z,x) |^2 dx.
\ee
\end{theorem}

\begin{lemma} \label{Ulimit}
Suppose Hypothesis \ref{hyp:limit}, let $H$ be some self-adjoint operator associated with
$\tau$ and let the boundary condition at $a$ be induced by $\Phi_0$. Denote by $U$ the associated
spectral transform from Section~\ref{sec:swm}. Then we have
\be \label{greenpsi}
(U\Psi(z,.))(\lam)=\frac{1}{\lam-z}
\ee
for every $z \in \C \setminus \sigma(H)$. Differentiating with respect to $z$ we even obtain
\be \label{greenpsidz}
(U\partial_z^k \Psi(z,.))(\lam)=\frac{k!}{(\lam-z)^{k+1}}.
\ee
\end{lemma}

We conclude this section by refining the integral representation of $M(z)$ which has been established in Theorem \ref{IntR}.

\begin{corollary}
Suppose the same assumptions as in Theorem \ref{limitthm}. Then we have
\be \label{intweyllimit}
M(z)= \re(M(\I)) + \int_{\R} \left( \frac{1}{\lam-z} - \frac{\lam}{1+\lam^2} \right)d\rho(\lam),
\ee
where $\rho$ (which is exactly the spectral measure from Section~\ref{sec:swm}) satisfies $\int_{\R} d\rho=\infty$
and $\int_{\R} \frac{d\rho(\lam)}{1+\lam^2}  < \infty$.
\end{corollary}

\section{Exponential growth rates}
\label{sec:egr}

It turns out that the real entire fundamental system $\Theta(z,x)$, $\Phi(z,x)$ from Section~\ref{sec:swm} is not sufficient for the proofs of our inverse uniqueness results.
To this end we will need information on the growth order of the functions $\Theta(\,\cdot\,,x)$ and $\Phi(\,\cdot\,,x)$.
Our presentation in this section will closely follow \cite[Section~3]{et}.

We will say a real entire solution $\Phi(z,x)$ is of growth order at most $s\geq 0$ if the entire functions $\Phi_1(\,\cdot\,,x)$ and $\Phi_2(\,\cdot\,,x)$ are of growth order at most $s$ for all $x\in I$.
 Our first aim is to extend Lemma~\ref{lem:pt} and to establish the connection between the growth order of $\Phi(z,x)$ and the convergence exponent
of the spectrum. We begin by recalling some basic notation and refer to the classical book by Levin \cite{lev} for proofs and further
background.

Given some discrete set $S\subseteq \C$, the number
\begin{equation}
 \inf\biggr\lbrace s\geq0 \,\biggr|\, \sum_{\mu\in S} \frac{1}{1+|\mu|^s}<\infty \biggr\rbrace \in [0,\infty],
\end{equation}
is called the convergence exponent of $S$. Moreover, the smallest $p\in\N_0$ such that 
\begin{equation}
 \sum_{\mu\in S} \frac{1}{1+|\mu|^{p+1}}<\infty
\end{equation}
will be referred to as the genus of $S$. Introducing the elementary factors
\begin{equation}
 E_p(\zeta,z) = \left(1-\frac{z}{\zeta}\right) \exp\left(\sum_{k=1}^p\frac{1}{k} \frac{z^k}{\zeta^k}\right), \quad z\in\C,
\end{equation}
if $\zeta\not=0$ and $E_p(0,z)=z$, we recall that the product $\prod_{\mu\in S} E_p(\mu,z)$
converges locally uniformly to an entire function of growth order $s$, where $s$ and $p$ are the
convergence exponent and genus of $S$, respectively.

\begin{theorem}\label{thm:IOphiev}
For each $s\geq 0$ the following properties are equivalent:
\begin{enumerate}[label=\emph{(\roman*)}, leftmargin=*, widest=XX]
\item The spectrum of $H^D_{(a,c)}$ is discrete and has convergence exponent at most $s$.
\item There  is a real entire solution $\Phi(z,x)$ of growth order at most $s$ which is non-trivial and lies in the domain of $H$ near $a$ for each $z\in\C$.
\end{enumerate}
In this case $s\ge 1$.
\end{theorem}

\begin{proof}
First suppose that the spectrum of $H^D_{(a,c)}$ is discrete and has convergence exponent at most $s$. 
Then the same 
holds true for the spectrum of the operator $H^N_{(a,c)}$ and as in \cite[Lemma~6.3]{kst2} one shows that $s\geq 1$.
 We will denote the spectra of these operators with
\begin{align*}
\sig(H^D_{(a,c)}) = \{ \mu_j \}_{j\in\Z} \quad\text{and}\quad \sig(H^N_{(a,c)}) = \{ \nu_{j} \}_{j\in\Z}.
\end{align*}
Note that the eigenvalues $\mu_j$, $\nu_{j}$, $j\in\Z$ are precisely the zeros
of $\Phi_1(\,\cdot\,,c)$ and $\Phi_2(\,\cdot\,,c)$, respectively. Also recall that both spectra are interlacing
(due to the fact that the quotient of the functions $\Phi_1(\,\cdot\,,c)$ and $\Phi_2(\,\cdot\,,c)$ is a Herglotz--Nevanlinna function)
\begin{align*}
\nu_{j-1} < \mu_j < \nu_j, \qquad j\in\Z,
\end{align*}
and that Krein's theorem \cite[Theorem~27.2.1]{lev} states
\begin{equation}\label{eqKrein}
m_-(z) = C \prod_{j\in\Z} \frac{E_0(\mu_j,z)}{E_0(\nu_{j},z)}
\end{equation}
for some real constant $C\not=0$.
Now consider the real entire functions
\[
\alpha(z) = \prod_{j\in\Z} E_p\left(\nu_{j},z\right) \quad\text{and}\quad \ti{\beta}(z) = \prod_{j\in\Z} E_p\left(\mu_{j},z\right),
\]
where $p\in\N_0$ is the genus of the sequences of eigenvalues.
Then clearly $\alpha(z)$ and $\ti{\beta}(z)$ are of growth order at most $s$ by Borel's theorem (see \cite[Theorem~4.3.3]{lev}).
Next note that $m_-(z) = \E^{h(z)} \ti{\beta}(z) \alpha(z)^{-1}$ for some entire function $h(z)$ since the right-hand side
has the same poles and zeros as $m_-(z)$. 
Comparing this with Krein's formula \eqref{eqKrein}, we obtain that $h(z)$ is in fact a polynomial of degree at most $p$:
\[
h(z) = \sum_{k=1}^p \frac{z^k}{k} \sum_{j\in\Z} \left( \frac{1}{\nu_{j}^k} - \frac{1}{\mu_j^k}\right) +\ln(C), \quad z\in\C,
\]
where the sums converge absolutely by the interlacing property of the eigenvalues.
In particular, the function $\beta(z) = -m_-(z)\alpha(z)= -\E^{h(z)} \ti{\beta}(z)$ is  of growth order at most $s$ as well.
Hence the solution $\Phi(z,x)$ with $\Phi_2(z,c) = \alpha(z)$ and $\Phi_1(z,c) = \beta(z)$, $z\in\C$ lies
in the domain of $H$ near $a$ and is of growth order at most $s$.

For the converse, let $\Phi(z,x)$ be a real entire solution of growth order at most $s$ which lies in the domain of $H$ near $a$.
Then $m_-(z) = -\Phi_2(z,c)/\Phi_1(z,c)$ and hence the spectrum of $H_{(a,c)}^D$ is discrete and coincides with the zeros of $\Phi_1(\,\cdot\,,c)$.
Now since the function $\Phi_1(\,\cdot\,,c)$ is of growth order at most $s$, its zeros are of convergence exponent at most $s$.
\end{proof}

Given a real entire solution $\Phi(z,x)$ of growth order $s\geq 1$ we are not able to prove the existence of a second solution $\Theta(z,x)$ of the same growth order.
However, we recall the following lemma which provides a criterion to ensure the existence of a second solution $\Theta(z,x)$ which has growth order arbitrarily close to $s$. To this end, we denote by $R_s(\C)$ the set of all entire functions $f(z)$ such that $|f(z)| \leq B \E^{A |z|^{s}}$ for some constants $A$, $B>0$. We will write $\Phi(\,\cdot\,,x) \in R_s(\C)$ if and only if $\Phi_1(\,\cdot\,,x), \Phi_2(\,\cdot\,,x) \in R_s(\C)$.

\begin{lemma}[\cite{et,kst2}]\label{lem:cor}
Let $s\geq1$ and suppose that $\Phi(\,\cdot\,,x) \in R_s(\C)$ for one (and hence for all) $x\in I$.
Then there is a real entire second solution $\Theta(z,x)\in R_s(\C)$ with
$W(\Theta(z),\Phi(z))=1$ if and only if
\begin{equation}\label{eqnSSphiboundlow}
   |\Phi_1(z,y)| + |\Phi_2(z,y)| \geq b \E^{-a|z|^{s}}, \quad z\in\C,
\end{equation}
for some constants $a$, $b>0$ and some $y\in I$.
\end{lemma}

The previous result enables us to provide a sufficient condition for a second solution of order $s+\eps$ to exist, in terms of the interlacing zeros of $\Phi_1(\,\cdot\,,c)$
and $\Phi_2(\,\cdot\,,c)$, which we denote by $\lbrace \mu_j\rbrace_{j\in\Z}$ and $\lbrace \nu_{j}\rbrace_{j\in\Z}$ respectively.

\begin{lemma}[\cite{et,kst2}]
Suppose $\Phi(z,x)$ is a real entire solution of growth order $s\geq 1$ and that for some $r>0$ all but finitely many of the discs given by
\begin{equation}\label{IOestmunu}
|z-\mu_j|<|\mu_j|^{-r} \quad\text{and}\quad |z-\nu_{j}|<|\nu_{j}|^{-r}, \quad j\in\Z,
\end{equation}
are disjoint. Then for every $\eps>0$ there is a real entire second solution $\Theta(z,x)$ with growth order at most $s+\eps$ and $W(\Theta(z),\Phi(z))=1$.
\end{lemma}

\begin{remark}\label{rem:uniqExp}
By the Hadamard product theorem~\cite[Theorem~4.2.1]{lev}, a solution $\Phi(z,x)$ of growth order $s\geq 1$ is unique up to a factor $\E^{g(z)}$,
for some polynomial $g(z)$ real modulo $\I\pi$ and of degree at most $s$.
A solution $\Theta(z,x)$ of growth order at most $s$ is unique only up to $f(z) \Phi(z,x)$, where $f(z)$ is a real entire function of growth order at most $s$.
\end{remark}

Finally, observe that under the assumptions in this section, one can always use the function $\hat{g}(z)=\exp(z^{2\ceil{(p+1)/2}})$ in Theorem~\ref{IntR}.

\section{A local Borg--Marchenko uniqueness result}
\label{sec:lbmt}

Now we turn to our Borg--Marchenko uniqueness result. Our argument will adapt the main strategy from
\cite{kst2} to the present situation (see also \cite{ben,cg}).
We will assume that our Dirac operator $H$ is in standard form, that is, normalized such that
\be
q_{\rm el} \equiv 0.
\ee
 First of all we note that by
\be
m_-(z) = - \frac{\Phi_2(z,c)}{\Phi_1(z,c)} = \I + \oo(1)
\ee
we see that $\Phi_1(z,c)$ and $\Phi_2(z,c)$ have the same asymptotic growth as $|z|\to\infty$ along nonreal rays and hence it does not
matter which one we take for this purpose.

\begin{lemma}[\cite{kst2}]\label{lemAsymM}
For each $x\in I$, the singular Weyl function $M(z)$ and the Weyl solution $\Psi(z,x)$ defined in~\eqref{defpsi}
have the following asymptotics:
\begin{align}\label{asymM}
M(z) &= -\frac{\Theta_j(z,x)}{\Phi_j(z,x)} + \OO\left(\frac{1}{\Phi_j(z,x)^2}\right),\\ \label{asympsi}
\Psi_j(z,x) &= \frac{\I}{2 \Phi_j(z,x)} \left( 1 + \oo(1) \right),
\end{align}
as $\im(z)\to\infty$.
\end{lemma}

In particular, \eqref{asymM} shows that asymptotics of $M(z)$ immediately follow once one
has corresponding asymptotics for the solutions $\Theta(z,x)$ and $\Phi(z,x)$. Moreover, the leading
asymptotics depend only on the values of $Q(x)$ near the endpoint $a$ (and on the choice of $\Theta(z,x)$ and $\Phi(z,x)$).
The following local Borg--Marchenko type uniqueness result shows that the converse is true as well.

In order to state this theorem, let $Q(x)$ and $\tilde{Q}(x)$ be two potentials on two intervals $(a,b)$ and $(a,\tilde{b})$, respectively.
By $H$ and $\tilde{H}$ we denote some corresponding self-adjoint operators with separated boundary conditions.
Furthermore, we will denote all quantities associated with $\tilde{H}$ by adding a twiddle. We will also use
the common short-hand notation $h_1(z) \sim h_2(z)$ to abbreviate the asymptotic relation
$h_1(z) = h_2(z) (1 + \oo(1))$ (or equivalently $h_2(z) = h_1(z) (1 + \oo(1))$) as $|z|\to\infty$
in some specified manner.

\begin{theorem}\label{thmbm}
Suppose $\Theta(z,x)$, $\tilde{\Theta}(z,x)$, $\Phi(z,x)$, $\tilde{\Phi}(z,x)$ are of growth order at most $s$ for some $s\geq 1$ and $\tilde{\Phi}_1(z,x) \sim \Phi_1(z,x)$ for one (and hence by~\eqref{IOphias} for all) $x\in(a,b)\cap(a,\ti{b})$ as $|z|\to\infty$ along some nonreal rays dissecting the complex plane into sectors of opening angles less than $\nicefrac{\pi}{s}$.
Then for each $c\in(a,b)\cap(a,\ti{b})$, the following properties are equivalent:
\begin{enumerate}[label=\emph{(\roman*)}, ref=(\roman*), leftmargin=*, widest=XX]
 \item\label{itbm1} We have $Q(x) = \tilde{Q}(x)$ for almost all $x\in(a,c)$ and $W_a(\Phi,\tilde{\Phi})=0$.
 \item\label{itbm2} For each $\delta>0$ there is an entire function $f(z)$ of growth order at most $s$ such that
  \begin{align*}\tilde{M}(z)-M(z) = f(z) + \OO\left(\frac{1}{\Phi_1(z,c)^{2}}\right),\end{align*}
   as $|z|\rightarrow\infty$ in the sector $|\im(z)|\geq \delta\,|\re(z)|$.
\end{enumerate}
\end{theorem}

\begin{proof}
If \ref{itbm1} holds, then by Remark~\ref{rem:uniqExp} the solutions are related by
\begin{align}\label{eqnLBMsolPHI}
 \tilde{\Phi}(z,x) = \Phi(z,x)\E^{g(z)}, \quad x\in(a,c],~ z\in\C,
\end{align}
and
\begin{align}\label{eqnLBMsolTHETA}
 \tilde{\Theta}(z,x) = \Theta(z,x) \E^{-g(z)} -f(z)\tilde{\Phi}(z,x), \quad x\in(a,c],~z\in\C,
\end{align}
 for some polynomial $g(z)$ of degree at most $s$ and some real entire function $f(z)$ of growth order at most $s$.
 From the asymptotic behavior of the functions $\Phi_1(z,x)$, $\tilde{\Phi}_1(z,x)$ we infer that $g=0$.
 Now the asymptotics in Lemma~\ref{lemAsymM} show that
 \begin{align*}
  \tilde{M}(z) - M(z) & = \frac{\Theta_1(z,c)}{\Phi_1(z,c)} - \frac{\tilde{\Theta}_1(z,c)}{\tilde{\Phi}_1(z,c)} + \OO\left(\frac{1}{\Phi_1(z,c)^{2}}\right) \\
                  & = f(z) + \OO\left(\frac{1}{\Phi_1(z,c)^{2}}\right),
 \end{align*}
 as $|z|\to\infty$ in any sector $|\im(z)| \ge \delta\, |\re(z)|$.
 
For the converse suppose that property \ref{itbm2} holds and for every fixed $x\in(a,c)$ and $j=1$, $2$ consider the entire function 
\begin{align*}
 G(z) = \tilde{\Phi}_j&(z,x)  \Theta_j(z,x) - \Phi_j(z,x) \tilde{\Theta}_j(z,x) - f(z)\Phi_j(z,x) \tilde{\Phi}_j(z,x), \quad z\in\C.
\end{align*}
Since away from the real axis this function may be written as
\begin{align*}
 G(z) & = \tilde{\Phi}_j(z,x) \Psi_j(z,x) - \Phi_j(z,x) \tilde{\Psi}_j(z,x) \\
      & \qquad\qquad + (\tilde{M}(z)-M(z) - f(z)) \Phi_j(z,x) \tilde{\Phi}_(z,x), \quad z\in\C\backslash\R,
\end{align*}
it vanishes as $|z|\to\infty$ along our nonreal rays. In fact, for the first two terms this
follows from \eqref{asympsi} together with our hypothesis that $\Phi_j(\,\cdot\,,x)$ and $\tilde{\Phi}_j(\,\cdot\,,x)$
have the same asymptotics. The last term tends to zero because of our assumption on the difference of the Weyl functions.
Moreover, by our hypothesis, $G$ is of growth order at most $s$ and thus we
can apply the Phragm\'en--Lindel\"of theorem (e.g., \cite[Section~6.1]{lev}) in the sectors bounded by our rays.
Thus $G$ is bounded on all of $\C$ and by Liouville's theorem it must be zero (since it vanishes along a ray); that is,
\[
\tilde{\Phi}_j(z,x) \Theta_j(z,x) - \Phi_j(z,x) \tilde{\Theta}_j(z,x) = f(z)\Phi_j(z,x)\tilde{\Phi}_j(z,x), \quad z\in\C.
\]
Dividing both sides of this identity by $\Phi_j(z,x)\tilde{\Phi}_j(z,x)$, differentiating with respect to $x$, and using $W(\Theta,\Phi)=W(\tilde{\Theta},\tilde{\Phi})=1$ shows 
\[
\frac{(-1)^j z - Q_{11}(x)}{\Phi_j(z,x)^2} = \frac{(-1)^j z - \tilde{Q}_{11}(x)}{\tilde{\Phi}_j(z,x)^2}, \quad z\in\C\backslash\R,
\]
for $j=1$, $2$ and almost all $x\in(a,c)$. 
Hence the poles on both sides must coincide, implying $\Phi_j(z,x)= \tilde{\Phi}_j(z,x)$ as well as $Q_{11}(x)=\tilde{Q}_{11}(x)$ for almost all $x\in(a,c)$.
But this also implies $W_a(\Phi,\tilde{\Phi})$ and $Q_{12}(x)=\tilde{Q}_{12}(x)$ finishing the proof.
\end{proof}

Note that the implication \ref{itbm2} $\Rightarrow$ \ref{itbm1} could also be proved under somewhat weaker conditions.
First of all the assumption on the growth of the entire functions $f(z)$ is only due to the use of the Phragm\'{e}n--Lindel\"{o}f principle.
Hence it would also suffice that for each $\eps>0$ we have
\begin{align}\label{eqnbmaltgrowth}
\sup_{|z|=r_n} |f(z)| \leq B \E^{A r_n^{s+\eps}},
\end{align}
for some increasing sequence of positive numbers $r_n\uparrow\infty$ and constants $A$, $B\in\R$.
Furthermore, for this implication to hold it would also suffice that the solutions have the same order of magnitude as $|z|\rightarrow\infty$ along our nonreal rays instead of the same asymptotics. 
Lastly, it would also be enough to only know the asymptotics of the difference of the Weyl functions in \ref{itbm2} along the nonreal rays. 

While at first sight it might look like the condition on the asymptotics of the solutions $\Phi_1(z,x)$ and $\tilde{\Phi}_1(z,x)$ requires knowledge
about them, this is not the case, since the high energy asymptotics will only involve some qualitative information
on the kind of the singularity at $a$.
Next, the appearance of the additional freedom of the function $f(z)$ just reflects the fact that we only ensure the same normalization
for the solutions $\Phi(z,x)$ and $\tilde{\Phi}(z,x)$ but not for $\Theta(z,x)$ and $\tilde{\Theta}(z,x)$ (cf.\ Remark~\ref{rem:uniqExp}).

\begin{corollary}\label{corbm}
Suppose $\Theta(z,x)$, $\tilde{\Theta}(z,x)$, $\Phi(z,x)$, $\tilde{\Phi}(z,x)$ are of growth order at most $s$ for some $s\geq 1$ and  $\tilde{\Phi}_1(z,x) \sim \Phi_1(z,x)$ for some $x\in(a,b)\cap(a,\ti{b})$ as $|z|\to\infty$ along some nonreal rays dissecting the complex plane into sectors of opening angles less than $\nicefrac{\pi}{s}$.
If
\begin{align}\label{eqncorbm}
 \tilde{M}(z) - M(z) = f(z), \quad z\in\C\backslash\R,
\end{align}
for some entire function $f(z)$ of growth order at most $s$, then $H= \tilde{H}$.
\end{corollary}

\begin{proof}
Without loss of generality, we may suppose that $b\leq \ti{b}$.
Then Theorem~\ref{thmbm} shows that $Q(x)=\tilde{Q}(x)$ for almost all  $x\in(a,b)$ and that the boundary condition at $a$ (if any) is the same.
As in the proof of Theorem~\ref{thmbm} one has~\eqref{eqnLBMsolPHI} as well as~\eqref{eqnLBMsolTHETA} with $g=0$ (note that the function $f$ in~\eqref{eqnLBMsolPHI} turns out to be the same as the one in ~\eqref{eqncorbm}) 
and hence
\begin{align*}
  \tilde{\Psi}(z,x) & = \tilde{\Theta}(z,x) + \tilde{M}(z)\tilde{\Phi}(z,x) 
               = \Theta(z,x) - f(z)\tilde{\Phi}(z,x) + (M(z)+f(z))\Phi(z,x) \\
              & = \Theta(z,x) + M(z)\Phi(z,x) = \Psi(z,x),
\end{align*}
for every $x\in(a,b)$ and $z\in\C\backslash\R$.
If $b<\tilde{b}$, then the right endpoint $b$ of $H$ would be regular as $\tilde{Q}$ is integrable over $[c,b]$.
Therefore, $\Psi(z,x)$ and thus also $\tilde{\Psi}(z,x)$ would satisfy some boundary condition at $b$, which is not possible. 
Hence we necessarily have $b= \tilde{b}$ and finally, since $\Psi(z,x)=\tilde{\Psi}(z,x)$, $H$ and $\tilde{H}$ also have the same boundary condition at $b$ (if any).
\end{proof}

Note that instead of assumption~\eqref{eqncorbm} it would also suffice to presume that for each fixed value $c\in(a,b)\cap(a,\tilde{b})$ one has 
\begin{align}
 M(z) - \tilde{M}(z) = f(z) + \OO\left(\frac{1}{\Phi_1(z,c)^2}\right),
\end{align}
as $|z|\rightarrow\infty$ along our nonreal rays and $\tilde{M}(z_0)= M(z_0)+f(z_0)$ for some $z_0\in\C\backslash\R$.

\section{Applications to perturbed radial Dirac operators}\label{secPertRD}

In this section we investigate perturbations of the free radial Dirac operator from Section \ref{sec4} with symmetric real-valued potentials $Q(x)$ on an interval $(0,b)$ satisfying 
\be\label{hyp:pertrad}
Q(x) = Q_\kappa(x) + P(x), \quad
\begin{cases}
P(x) \in L^1_{loc}[0,b), & \kappa \ne \frac{1}{2},\\
(1+|\log(x)|) P(x) \in L^1_{loc}[0,b), & \kappa=\frac{1}{2}.
\end{cases}
\ee
For notational simplicity, we choose $m=0$ as it can be absorbed in $P(x)$ anyway. 

Under these assumptions, the regular solution of $\tau u = z u$ will satisfy
\be\label{eqvarconst}
\Phi(z,x) = \Phi_\kappa(z,x) + \int_0^x K(z,x,y) P(y) \Phi(z,y) dy, 
\ee
where
\be
K(z,x,y) = \begin{pmatrix}
z K_\kappa(z^2,x,y) & \left(- \frac{\partial}{\partial x} + \frac{\kappa}{x}\right) K_{\kappa-1}(z^2,x,y)\\
\left(\frac{\partial}{\partial x} + \frac{\kappa}{x}\right) K_\kappa(z^2,x,y) & z K_{\kappa-1}(z^2,x,y)
\end{pmatrix}
\ee
and
\be
K_l(z^2,x,y) = \begin{cases}
\phi_l(z^2,x) \theta_l(z^2,y) - \phi_l(z^2,y) \theta_l(z^2,x),& l\ge -\frac{1}{2},\\
-K_{-1-l}(z^2,x,y), & l\in [-1,-\frac{1}{2}).
\end{cases}
\ee
For its investigation we will need the following standard estimates (see, e.g., Appendix A in \cite{kst}).

\begin{lemma}[\cite{kst}]\label{lem:estbes}
For every $l> - \frac{1}{2}$ there is a constant $C$ such that for every $z\in\C$ and $0 < y \leq x\le 1$ the following estimates hold:
\begin{align}\label{estphil}
\left|\phi_l(z^2,x)\right| &\leq C \left(\frac{x}{1+ |z| x}\right)^{l+1} \E^{|\im(z)| x},\\
\left| \frac{\partial}{\partial x} \phi_l(z^2,x)\right| &\leq C \left(\frac{x}{1+ |z| x}\right)^{l} \E^{|\im(z)| x},\\ \label{estGl}
\left|K_l(z^2,x,y)\right| &\leq C \left(\frac{x}{1+ |z| x}\right)^{l+1} \left(\frac{1+ |z| y}{y}\right)^l \E^{|\im(z)| (x-y)},\\
\left|\frac{\partial}{\partial x} K_l(z^2,x,y)\right| &\leq C \left(\frac{x}{1+ |z| x}\right)^l \left(\frac{1+ |z| y}{y}\right)^l \E^{|\im(z)| (x-y)}.
\end{align}
For the case $l=-\frac{1}{2}$, one has to replace the estimates for $K_l$ by
\begin{align}\label{a8}
\left|K_{-1/2}(z^2,x,y)\right| &\leq C \left(\frac{x}{1+ |z| x}\right)^{1/2} \left(\frac{y}{1+ |z| y}\right)^{1/2}
\\ \nn & \qquad\qquad\qquad\qquad \times
\E^{|\im(z)| (x-y)}(1-\log(y)),\\
\left|\frac{\partial}{\partial x}K_{-1/2}(z^2,x,y)\right| & \leq C \left(\frac{1+ |z| x}{x}\right)^{1/2} \left(\frac{y}{1+ |z| y}\right)^{1/2}
\\ \nn & \qquad\qquad\qquad\qquad \times
\E^{|\im(z)| (x-y)} (1-\log(y)).
\end{align}
\end{lemma}

Based on these estimates, we can solve \eqref{eqvarconst} in the usual way. 

\begin{lemma}\label{lemPRDPhi}
If the potential $Q(x)$ is of the form~\eqref{hyp:pertrad}, then the integral equation \eqref{eqvarconst} has a unique solution satisfying
\begin{equation}\label{eq:estphi}
|\Phi(z,x)-\Phi_\kappa(z,x)|\le C\,  I(x) \left(\frac{x}{1+ |z| x}\right)^{\kappa} \E^{|\im(z)| x}
\end{equation}
near zero for some constant $C>0$, where 
\be
I(x) = \int_0^x \|P(r)\| \begin{cases}
 dr, & \kappa \ne \frac{1}{2},\\
 (1-\log(r)) dr, & \kappa = \frac{1}{2}.
\end{cases}
\ee
Hereby, $\|.\|$ is the operator norm corresponding to the euclidean norm on $\C^2$. 
\end{lemma}

\begin{proof}
 Using the estimates from Lemma~\ref{lem:estbes}, we immediately get 
\begin{align*}
\left|\Phi_{\kappa}(z,x)\right| &\leq C \left(\frac{x}{1+ |z| x}\right)^{\kappa} \E^{|\im(z)| x}
\end{align*}
for some constant $C>0$, all $z\in\C$ and $0<x\leq 1$ as well as
\begin{align*}
 \left\|K(z,x,y)\right\| &\leq C \left(\frac{x}{1+ |z| x}\right)^{\kappa} \left(\frac{1+ |z| y}{y}\right)^{\kappa} \E^{|\im(z)| (x-y)} \begin{cases} 1, & \kappa\neq\frac{1}{2}, \\
(1-\log(y)),& \kappa=\frac{1}{2}, 
\end{cases}
\end{align*}
 for all $z\in\C$ and $0<y\leq x \leq 1$. 
Hereby, we have also employed the inequality $\frac{y}{1+|z|y}\le \frac{x}{1+|z|x}$ in the case $\kappa\in [0,\frac{1}{2})$. 
Now we can construct $\Phi(z,x)$ using the standard iterative procedure. Namely, let
\begin{align}\label{eq:Phin}
\Phi(z,x) & =\sum_{n=0}^\infty\Phi^n(z,x), & \Phi^n(z,x) & =\int_0^x K(z,x,y) P(y)\Phi^{n-1}(z,y)dy,
\end{align}
and $\Phi^0(z,x)=\Phi_\kappa(z,x)$. 
Using induction and the estimates from above one shows 
\begin{align*}
\left|\Phi^n(z,x)\right| \le \frac{C^{n+1}}{n!}\left(\frac{x}{1+ |z| x}\right)^{\kappa} \E^{|\im(z)| x} I(x)^n
\end{align*}
near zero. 
Therefore, the sum in~\eqref{eq:Phin} converges uniformly near zero (and also locally uniformly in $z\in\C$) to a solution of the integral equation~\eqref{eqvarconst} with the required estimate~\eqref{eq:estphi}. 
\end{proof}

Since it is readily verified that a solution of the integral equation~\eqref{eqvarconst} is a solution of $\tau u = zu$, Lemma~\ref{lemPRDPhi} gives rise to a real entire solution $\Phi(z,x)$. 
In the following, we will also need the asymptotics of this solution as $\im(z)\to\infty$. 
The next result was first shown in \cite{ser} for the special case $\kappa\in\N_0$.
We give a streamlined proof which works for every $\kappa\geq 0$.

\begin{lemma}\label{lem:asphi}
If the potential $Q(x)$ is of the form in~\eqref{hyp:pertrad} with its trace normalized to zero, then the solution of \eqref{eqvarconst} has the asymptotics
\be\label{asymPhiprd}
|\Phi(z,x)-\Phi_\kappa(z,x)| = \oo\bigl(|z|^{-\kappa}\E^{|\im(z)| x}\bigr)
\ee
as $|z|\to\infty$ for all $x$ near zero. In particular, 
\begin{align}
\Phi(z,x) \sim  
(\pm z)^{-\kappa}\begin{pmatrix}\sin\bigl(z x \mp \frac{\kappa \pi}{2}\bigr)\\ \cos\bigl(z x \mp \frac{\kappa \pi}{2}\bigr)\end{pmatrix}
\end{align}
as $|z|\to\infty$ in any sector $|\arg(\pm z)| < \pi - \delta$.
\end{lemma}

\begin{proof}
We will first establish the claim for the case when $P$ has a bounded continuous derivative near zero. 
More precisely, we suppose that 
\begin{align}\label{eqnPlog}
\begin{cases}
P(x)\in C^1[0,x], & \kappa\neq \frac{1}{2},\\
(1-\log(x))P(x),\ (1-\log(x))P'(x) \in C[0,x], & \kappa=\frac{1}{2}.
\end{cases}
\end{align}
Firstly, let us estimate the $\Phi^1(z,x)$ defined in \eqref{eq:Phin} by decomposing it in 
\[
\Phi^1(z,x)= \Pi(z,x)+ \Upsilon(z,x),
\]
where
\begin{align*}
\Pi(z,x)&=\int_0^x P_{11}(y)K(z,x,y)\begin{pmatrix}
\Phi_{\kappa,1}(z,y)\\
-\Phi_{\kappa,2}(z,y)
\end{pmatrix}dy,\\
\Upsilon(z,x)&=\int_0^x P_{12}(y)K(z,x,y)\begin{pmatrix}
\Phi_{\kappa,2}(z,y)\\
\Phi_{\kappa,1}(z,y)
\end{pmatrix}dy.
\end{align*}

In order to estimate $\Pi(z,x)$, we integrate by parts  to obtain 
\[
\Pi(z,x)=P_{11}(x)R(z,x,x)-\int_0^x P_{11}'(y) R(z,x,y)dy,
\]
 where
 \[
R(z,x,y) =\int_0^y K(z,x,t)\begin{pmatrix}
\Phi_{\kappa,1}(z,t)\\
-\Phi_{\kappa,2}(z,t)
\end{pmatrix}dt.
 \] 
Using the identities
\begin{align}\label{eq:k12}
K_{12}(z,x,y)&=a_\kappa(x)K_{\kappa-1}(z^2,x,y)=a_\kappa^*(y)K_{\kappa}(z^2,x,y),\\
K_{21}(z,x,y)&=a_\kappa^\ast(x) K_{\kappa}(z^2,x,y)=a_\kappa(y)K_{\kappa-1}(z^2,x,y),\label{eq:k21}
\end{align}
as well as the relations provided in Section~\ref{sec4}, we compute  
\begin{align*}
R(z,x,y)=\begin{pmatrix}
-K_\kappa(z^2,x,y)a^*_\kappa\phi_{\kappa}(z^2,y)+\lim_{t\to 0}K_\kappa(z^2,x,t)a^*_\kappa\phi_{\kappa}(z^2,t)\\
-zK_{\kappa-1}(z^2,x,y)\phi_{\kappa}(z^2,y)+\lim_{t\to 0}zK_{\kappa-1}(z^2,x,t)\phi_{\kappa}(z^2,t)
\end{pmatrix}.
\end{align*}
Noting that (cf.\ \cite[Section 2]{kt})
\[
\phi_l(\zeta,x)\sim C_l^{-1} x^{l+1},\quad \theta_l(\zeta,x)\sim \frac{C_l x^{-l}}{2l+1}, \quad C_l=\frac{2^{l+1}\Gamma(l+\frac{3}{2})}{\sqrt{\pi}},\quad l> -\frac{1}{2},
\]
as $x\to 0$, 
we get
\begin{align*}
R(z,x,y)=\begin{pmatrix}
-K_\kappa(z^2,x,y)a^*_\kappa\phi_{\kappa}(z^2,y)\\ 
-zK_{\kappa-1}(z^2,x,y)\phi_{\kappa}(z^2,y)
\end{pmatrix}+\begin{pmatrix}
\phi_{\kappa}(z^2,x)\\
0
\end{pmatrix}
\end{align*} 
and, in particular, 
\[
R(z,x,x)=\begin{pmatrix}
\phi_{\kappa}(z^2,x)\\
0
\end{pmatrix}.
\]
From all this we immediately infer  
\begin{align*}
 |R(z,x,y)| & \le C \left(\frac{x}{1+|z|x}\right)^{\kappa+1} \E^{|\im(z)|x}\begin{cases}
1, & \kappa\neq\frac{1}{2},\\
1-\log(y), & \kappa=\frac{1}{2},
\end{cases}
\end{align*}
for all $0<y\le x\le 1$. 
Hence we arrive at the following estimate for $x$ near zero 
\begin{align*}
|\Pi(z,x)|\le C B_P \left(\frac{x}{1+|z|x}\right)^{\kappa+1}\E^{|\im(z)|x},
\end{align*}
where the constant $B_P$ depends on the potential as follows 
\begin{align*}
 B_P = \begin{cases} \|P(y)\|_{C^1[0,x]}, & \kappa\not=\frac{1}{2}, \\ \|(1-\log(y))P(y)\|_{C[0,x]} + \|(1-\log(y)) P'(y)\|_{C[0,x]}, & \kappa = \frac{1}{2}. \end{cases}
\end{align*}

Similarly, in order to estimate $\Upsilon(z,x)$, consider the function
\[
S(z,x,y)=\int_0^y K(z,x,t)\begin{pmatrix}
\Phi_{\kappa,2}(z,t)\\
\Phi_{\kappa,1}(z,t)
\end{pmatrix}dt.
\]
Hereby note that the integrand may be written as  
\begin{align*}
& \begin{pmatrix}
z K_{\kappa}(z^2,x,t)a^*_\kappa\phi_{\kappa}(z^2,t)  + z a_\kappa(x)K_{\kappa-1}(z^2,x,t)\phi_{\kappa}(z^2,t)\\
a_{\kappa}^*(x) K_{\kappa}(z^2,x,t)a^*_\kappa\phi_{\kappa}(z^2,t)  + z^2 K_{\kappa-1}(z^2,x,t)\phi_{\kappa}(z^2,t)
\end{pmatrix}\\
& \qquad\qquad =\begin{pmatrix}
z (K_{\kappa}(z^2,x,t)a^*_\kappa\phi_{\kappa}(z^2,t)  +  a_\kappa^*(t)K_{\kappa}(z^2,x,t)\phi_{\kappa}(z^2,t))\\
a_{\kappa}(t) K_{\kappa-1}(z^2,x,t)a^*_\kappa\phi_{\kappa}(z^2,t)  +  K_{\kappa-1}(z^2,x,t)a_\kappa a_\kappa^*\phi_{\kappa}(z^2,t)
\end{pmatrix}.
\end{align*}
Using the identities \eqref{eq:k12} and  \eqref{eq:k21} once more, 
we compute
\begin{align*}
S(z,x,y) & =\begin{pmatrix}
zK_{\kappa}(z^2,x,y)\phi_{\kappa}(z^2,y)\\
-K_{\kappa-1}(z^2,x,y) a^*_\kappa\phi_{\kappa}(z^2,y)
\end{pmatrix}  \\
 & \qquad\qquad + \int_0^y \frac{2\kappa}{t}\begin{pmatrix}
zK_{\kappa}(z^2,x,t)\phi_{\kappa}(z^2,t)\\
K_{\kappa-1}(z^2,x,t) a^*_\kappa\phi_{\kappa}(z^2,t)
\end{pmatrix}dt
\end{align*}
and, in particular,  
\[
S(z,x,x) =
\int_0^x \frac{2\kappa}{t}\begin{pmatrix}
zK_{\kappa}(z^2,x,t)\phi_{\kappa}(z^2,t)\\
K_{\kappa-1}(z^2,x,t)a^*_\kappa\phi_{\kappa}(z^2,t)
\end{pmatrix}dt.
\]
Now using the estimates from Lemma \ref{lem:estbes}, we obtain
\begin{align*}
|S(z,x,y)|&\le C \frac{T(z,x)}{|z|} \left(\frac{x}{1+ |z| x}\right)^{\kappa} \E^{|\im(z)| x}  \begin{cases} 1, & \kappa\not=\frac{1}{2}, \\ 1-\log(y), & \kappa=\frac{1}{2},\end{cases} 
\end{align*}
where the function $T(z,x)$ is given by 
\begin{align*}
 T(z,x) = \begin{cases} \log(1+|z|x), & \kappa\not=\frac{1}{2}, \\ \log(1+|z|x) + |\mathrm{Li}_2(-|z|x)|, & \kappa=\frac{1}{2}. \end{cases}
\end{align*}
Hereby note that $\mathrm{Li}_2(-|z|x)=\oo(|z|)$ as $|z|\to \infty$ (see \cite[formula (25.12.2)]{dlmf}). Using these estimates and integrating by parts we arrive at the following estimate  
\begin{align*}
|\Upsilon(z,x)|\le C B_P \frac{T(z,x)}{|z|} \left(\frac{x}{1+ |z| x}\right)^{\kappa} \E^{|\im(z)| x}. 
\end{align*} 
The latter immediately implies the corresponding estimate for $\Phi^1(z,x)$. 

Noting that the functions in~\eqref{eqnPlog} are continuous, we infer from \eqref{eq:Phin}  
\begin{align}\label{eq:Phinest}
|\Phi^{n}(z,x)| & \le \frac{C^{n} B_P^{n}}{(n-1)!} \frac{T(z,x)}{|z|} \left(\frac{x}{1+|z|x}\right)^{\kappa}\E^{|\im(z)|x} 
\end{align}
for all $n\geq 2$. 
Summing up these inequalities, we arrive at the estimate 
\begin{align*}
|\Phi(z,x)-\Phi_\kappa(z,x)| & \le C \E^{2 B_P} \frac{T(z,x)}{|z|} \left(\frac{x}{1+|z|x}\right)^{\kappa}\E^{|\im(z)|x} 
\end{align*}
for some constant $C$, yielding the claim in this case. 

 Now we return to the general case and fix some $x\in(0,b)$ near zero. 
 Then for every $\varepsilon>0$ there is a $P_\varepsilon$ with a bounded continuous first derivative on $(0,x)$ such that $\| (P(y)-P_\varepsilon(y))(1-\log(y))\|_{L^1(0,x)}\le \varepsilon$.
 Hence, we infer from \eqref{eq:Phin} that 
\[
|\Phi^1(z,x;P)-\Phi^1(z,x;P_\varepsilon)|\le \varepsilon \, C\left(\frac{x}{1+|z|x}\right)^{\kappa}\E^{|\im(z)|x},
\]
as $|z|\to \infty$. Employing \eqref{eq:Phinest}, we arrive at the following estimate
\[
|\Phi^1(z,x;P)|\le C\left(\frac{x}{1+|z|x}\right)^{\kappa}\left(\varepsilon +B_{P_\varepsilon} \frac{T(z,x)}{|z|}\right)\E^{|\im(z)|x}.
\]
Using this estimate, after iteration we obtain
\[
|\Phi^n(z,x;P)|\le \frac{C^n B_P^{n-1}}{(n-1)!}\left(\frac{x}{1+|z|x}\right)^{\kappa}\left(\varepsilon + B_{P_\varepsilon} \frac{T(z,x)}{|z|}\right)\E^{|\im(z)|x},
\] 
which implies \eqref{asymPhiprd} since $\varepsilon>0$ can be chosen arbitrarily small.

Finally, utilizing the asymptotics (cf.\ \cite[Sections~VII.21 and VII.22]{wat})
\begin{align*}
\phi_\kappa(z^2,x)  = 
(\pm z)^{-\kappa-1}\sin\bigl(\pm z x- \frac{\kappa \pi}{2}\bigr) + \OO\bigl(|z|^{-\kappa-2}\E^{|\im(z)| x}\bigr),
\end{align*}
as $\lam\to\pm\infty$ as $|z|\to\infty$ in any sector $|\arg(\pm z)| < \pi -\delta$, we obtain
\begin{align*}
\Phi_\kappa(z,x)= 
(\pm z)^{-\kappa}\begin{pmatrix}\sin\bigl(z x \mp \frac{\kappa \pi}{2}\bigr)\\ \cos\bigl(z x \mp \frac{\kappa \pi}{2}\bigr)\end{pmatrix}
 + \OO\bigl(|z|^{-\kappa-1}\E^{|\im(z)| x}\bigr),
\end{align*}
from which the very last claim follows.
\end{proof}

Now we are ready to prove our main results in this section.

\begin{theorem}\label{thmpbgn}
Suppose that the potential $Q(x)$ is of the form~\eqref{hyp:pertrad} and let $\Phi(z,x)$ be as in Lemma~\ref{lemPRDPhi}.
Then there is a second solution $\Theta(z,x)$ with $W(\Theta(z),\Phi(z))=1$ such that the corresponding singular Weyl function is given by 
\be
M(z) = (1+z^2)^{\ceil{\kappa}} \int_\R \left(\frac{1}{\lam-z} - \frac{\lam}{1+\lam^2}\right) \frac{d\rho(\lam)}{(1+\lam^2)^{\ceil{\kappa}}}, \quad z\in\C\backslash\R. 
\ee
In particular, $M(z)\in N_{\kappa_0}^\infty$ for some $\kappa_0 \le \ceil{\kappa}$.
\end{theorem}

\begin{proof}
By Lemma~\ref{lem:asphi} we have $|\Phi(\lam,x)|^2  \sim |\lam|^{-2\kappa}$ as $\lambda\rightarrow \pm\infty$. 
Since the function $|\Phi(\lambda,x)|^2$ belongs to $L^1(\R, (1+\lam^2)^{-1}d\rho)$ by Lemma~\ref{lemUub}, so does $(1+\lambda^2)^{-\ceil{\kappa}}$
and the claim follows from Theorem~\ref{thm:nkap}.
\end{proof}

Using \eqref{asymPhiprd}, we can even strengthen Lemma~\ref{lem:cor} in
this special case.

\begin{lemma}\label{lem:bes}
Suppose that the potential $Q(x)$ is of the form~\eqref{hyp:pertrad} and let $\Phi(z,x)$ be as in Lemma~\ref{lemPRDPhi}. 
Then there is a second solution $\Theta(z,x)$ with
$W(\Theta(z),\Phi(z))=1$ and such that $\Theta(.,x)$ is of finite exponential type for every $x\in(0,b)$.
\end{lemma}

\begin{proof}
The estimate in Lemma~\ref{lem:asphi} imply the asymptotics 
\[
|\Phi_2(z,x) - \I \Phi_1(z,x)| = |z|^{-\kappa} \E^{\im(z) x}
+ \oo\bigl(|z|^{-\kappa}\E^{\im(z|x}\bigr)
\]
as $|z|\to \infty$ in $\im(z)\ge 0$.  Since $\Phi_2(z,x) - \I\Phi_2(z,x)$ has no zeros in the closed upper half plane, this shows
\[
|\Phi_2(z,x)| + |\Phi_1(z,x)| \ge |\Phi_2(z,x) -\I \Phi_1(z,x)| \geq  \frac{c}{1+|z|^\kappa} ,\quad \im(z)\ge 0,
\]
for some $c >0$. Conjugating $z$ shows that the same estimate holds for $\im(z)\le 0$ and thus the claim is a consequence of Lemma~\ref{lem:cor}. 
Hereby, also note that $\Theta(.,x)$ is of finite exponential type for every $x\in(0,b)$ if and only if it is for some $x\in(0,b)$. 
\end{proof}

In particular, the results of the previous sections apply to this example. For instance, Theorem~\ref{thmbm} now takes the following form.

\begin{theorem}\label{thmBMPRD}
Let $Q(x)$ and $\tilde{Q}(x)$ be two potentials of the form~\eqref{hyp:pertrad} with the same $\kappa\geq 0$ and their trace normalized to zero. 
Choose the solutions $\Phi(z,x)$, $\tilde{\Phi}(z,x)$ as in Lemma~\ref{lemPRDPhi}, $\Theta(z,x)$, $\tilde{\Theta}(z,x)$ according to Lemma~\ref{lem:bes}  and let $c\in(0,b)\cap(0,\tilde{b})$. 
If for every $\eps>0$ there is an entire function $f(z)$ of finite exponential type such that 
\begin{align}
 \tilde{M}(z)-M(z) = f(z) + \OO\big(\E^{-2(c-\eps) |\im(z)|}\big)
\end{align}
as $z\to\infty$ along the imaginary axis, then $Q(x)=\tilde{Q}(x)$ for almost all $x\in(0,c)$.
\end{theorem}

 To be precise, the claim of Theorem~\ref{thmBMPRD} is somewhat stronger than the one in Theorem~\ref{thmbm} and does not immediately follow from there. 
 However, due to the fact that the exponential types of our solutions $\Phi(z,x)$ and $\Theta(z,x)$ are known to be finite, a simple modification of the proof of Theorem~\ref{thmbm} yields the claim. 

For further results concerning the (inverse) spectral theory of radial Dirac operators see \cite{ahm,ahm2,ahm3,lema,ma,ma2,mp,pu,ser,st,teosc}

\section{Uniqueness results for operators with discrete spectra}
\label{sec:urds}

Now we are finally able to investigate when the spectral measure determines the potential for operators with purely discrete spectrum.
In this respect, observe that the uniqueness results for the singular Weyl function from the previous sections do not  immediately yield
such results. In fact, if $\rho= \tilde{\rho}$, then the difference of the corresponding singular Weyl functions is an entire function by Theorem~\ref{IntR}.
However, in order to apply Corollary~\ref{corbm} we would need some bound on the growth order of this function.
Fortunately, in the case of purely discrete spectrum with finite convergence exponent, refinements of the arguments in the proof
of Theorem~\ref{thmbm} show that the growth condition is superfluous.
We continue to assume that our Dirac operators are in standard form, that is, normalized such that
\be
q_{\rm el} \equiv 0.
\ee

\begin{corollary}[\cite{et}]\label{corbmdis}
Suppose $\Phi(z,x)$, $\tilde{\Phi}(z,x)$ are of growth order at most $s$ for some $s\geq 1$ and  $\tilde{\Phi}(z,x) \sim \Phi(z,x)$ for an $x\in(a,b)\cap(a,\tilde{b})$ as $|z|\to\infty$ along some nonreal rays dissecting the complex plane into sectors of opening angles less than $\nicefrac{\pi}{s}$.
Furthermore, assume that $H$ and $\tilde{H}$ have purely discrete spectra with convergence exponent at most $s$.
If
\begin{equation}
 \tilde{M}(z) - M(z) = f(z), \quad z\in\C\backslash\R,
\end{equation}
for some entire function $f(z)$, then $H=\tilde{H}$.
\end{corollary}

Now the lack of a growth restriction in Corollary~\ref{corbmdis} implies that it immediately translates into a corresponding uniqueness result for the spectral measure.

\begin{theorem}[\cite{et}]\label{thmSpectFuncDisc}
Suppose that $\Phi(z,x)$, $\tilde{\Phi}(z,x)$ are of growth order at most $s$ for some
$s\geq 1$ and $\tilde{\Phi}(z,x)\sim\Phi(z,x)$ for an $x\in(a,b)\cap(a,\tilde{b})$ as $|z|\rightarrow\infty$ along some
nonreal rays dissecting the complex plane into sectors of opening angles less than $\nicefrac{\pi}{s}$.
Furthermore, assume that $H$ and $\tilde{H}$ have purely discrete spectra with convergence exponent at most $s$.
If the corresponding spectral measures $\rho$ and $\tilde{\rho}$ are equal, then we have $H=\tilde{H}$.
\end{theorem}

\begin{proof}
Since the spectral measures are the same, Theorem~\ref{IntR} shows that the difference of the corresponding singular Weyl functions is an entire function and Corollary~\ref{corbmdis} is applicable.
\end{proof}

We remark that similar results can be proven using the theory of de Branges spaces \cite{je,ekt}. However, our assumptions on $\Phi(z,x)$ here are of a different nature and more convenient to derive in certain situations; \cite[Section~6]{et}.  
Nevertheless, in some sense our assumptions here are stronger since they exclude (in the case $I=\R$) the possibility that one potential is a translation of the other one 
(which clearly would leave the spectral measure invariant).

Also note that in the case of discrete spectra, the spectral measure is uniquely determined by the eigenvalues $\lambda_n$ and the corresponding norming constants
\begin{equation}
 \gamma_n^2 = \int_a^b \left|\Phi(\lambda_n,x)\right|^2 dx,
\end{equation}
since in this case we have
\begin{equation}
 \rho = \sum_n \gamma_n^{-2} \delta_{\lambda_n},
\end{equation}
where $\delta_{\lambda}$ is the unit Dirac measure in the point $\lambda$.

As another application, we are also able to prove a generalization of Hochstadt--Lieberman type uniqueness results.
To this end, let us consider an operator $H$ whose spectrum is purely discrete and has convergence exponent (at most) $s$.
Since the operator $H^D_c = H^D_{(a,c)} \oplus H^D_{(c,b)}$ with an additional Dirichlet boundary condition at $c$ is
a rank one perturbation of $H$, we conclude that the  convergence exponents of both $H^D_{(a,c)}$ and $H^D_{(c,b)}$ are
at most $s$ and hence by Theorem~\ref{thm:IOphiev}, there are real entire solutions $\Phi(z,x)$ and $\Pi(z,x)$ of growth order at most $s$ which belong to
 the domain of $H$ near $a$ and $b$, respectively.

\begin{theorem}[\cite{et}]\label{thmHL}
Suppose $H$ is an operator with purely discrete spectrum of finite convergence exponent $s$. Let $\Phi(z.x)$ and $\Pi(z,x)$ be
entire solutions of growth order at most $s$ which lie in the domain of $H$ near $a$ and $b$, respectively, and suppose there is a $c\in I$ such that
\begin{equation}\label{eqeahl}
\frac{\Pi_1(z,c)}{\Phi_1(z,c)}=\OO(1),
\end{equation}
as $|z|\rightarrow\infty$ along some nonreal rays dissecting the complex plane into sectors of opening angles less than $\nicefrac{\pi}{s}$.
 Then every other isospectral operator $\tilde{H}$ for which $\tilde{Q}(x)=Q(x)$ almost everywhere on $(a,c)$ and which is associated with the same boundary
condition at $a$ (if any) is equal to $H$.
\end{theorem}

Note that by \eqref{IOphias} the growth of $\Phi(\,\cdot\,,c)$ will increase as $c$ increases while (by reflection) the growth of $\Pi(\,\cdot\,,c)$
will decrease. In particular, if the bound \eqref{eqeahl} holds for some $c$, then it will hold for any other $c'>c$ as well.

As an example, we give a generalization of the Hochstadt--Lieberman result from \cite{hl} to Dirac operators on $(0,1)$ with
radial-type singularities at both endpoints.

\begin{theorem}
Let $\eta$, $\kappa\ge 0$ and consider an operator of the form
\begin{align}
H & = \frac{1}{\I} \sig_2 \frac{d}{dx} + Q(x), & Q(x) & = \begin{pmatrix}m & \frac{\kappa}{x} + \frac{\eta}{1-x}\\
\frac{\kappa}{x} + \frac{\eta}{1-x} &-m\end{pmatrix} + P(x),
\end{align}
on the interval $(0,1)$ such that 
\begin{align}
 f_\kappa(x) f_\eta(1-x) P(x) & \in L^1(0,1), & f_\kappa(x) & = \begin{cases}1, & \kappa\not=\frac{1}{2}, \\ 1 - \log(x), & \kappa=\frac{1}{2}.  \end{cases} 
\end{align}
If $0\le \kappa <\nicefrac{1}{2}$, then we choose the boundary condition at zero which is induced by the solutions $\Phi(z,x)$ from Section~\ref{secPertRD}. 
Suppose $\tilde{H}$ satisfies $\tilde{Q}(x)=Q(x)$ for $x\in(0,\nicefrac{1}{2}+\eps)$ and has the same boundary condition at zero if $0\le \kappa <\nicefrac{1}{2}$,
where $\eps=0$ if $\eta\ge \kappa$ and $\eps>0$ if $\eta<\kappa$. Then, $H=\tilde{H}$ if both have the same spectrum.
\end{theorem}

\begin{proof}
Immediate from Theorem~\ref{thmHL} together with the asymptotics of solutions for $H$ given in Lemma~\ref{lem:asphi}.
\end{proof}

\bigskip
\noindent
{\bf Acknowledgments.}
We thank Fritz Gesztesy and Alexander Sakhnovich for hints with respect to the literature.
J.E.\ and G.T.\ gratefully acknowledge the stimulating atmosphere at the {\em Institut Mittag-Leffler} during spring 2013 where parts of this paper were written during the international research program on {\em Inverse Problems and Applications}.

\end{document}